\documentclass[11pt]{article}
\usepackage{amsfonts}
\usepackage{amscd}
\usepackage{amsmath,amstext,amsthm,amsbsy,amssymb}
\usepackage{mathrsfs}
\usepackage{multirow}
\usepackage{epsfig}
\newcommand{\bi}{{\bf i}}
\newcommand{\bj}{{\bf j}}
\newcommand{\bk}{{\bf k}}
\newcommand{\bc}{{\mathbb C}}
\newcommand{\br}{{\mathbb R}}
\newcommand{\bh}{{\mathbb H}}

\newtheorem{thm}{Theorem}[section]
\newtheorem{lem}{Lemma}[section]
\newtheorem{pro}{Proposition}[section]

\newtheorem{exam}{Example}[section]
\newtheorem{defi}{Definition}[section]

\begin{document}

\title{Quadratic equation  in  split  quaternions}
\author{ Wensheng Cao \\
School of Mathematics and Computational Science,\\
Wuyi University, Jiangmen, Guangdong 529020, P.R. China\\
e-mail: {\tt wenscao@aliyun.com}}
\date{}
\maketitle

\bigskip
{\bf Abstract} \,\,  Generally  speaking, it is difficult to solve equations in  algebras which is
noncommutative and contains nontrivial zero divisors. Split quaternion algebra  plays an important  role in a modern
physics, however, it is noncommutative and contains nontrivial zero divisors. In this paper,  we derive
explicit formulas for computing the roots of $ax^{2}+bx+c=0$ in split  quaternion algebra.

\vspace{1mm}\baselineskip 12pt



\section{Introduction}
\subsection{Split quaternions}
Let $\br$ and $\bc$ be the field of  real and  complex numbers, respectively.
The split quaternion algebra is a non-commutative extension of the complex numbers.
Such an  algebra is a 4-dimensional associative algebra original introduced by James Cockle [4] in 1849.
Split quaternions can be represented as
$$\bh_s=\{x=x_0+x_1\bi+x_2\bj+x_3\bk,x_i\in \br,i=0,1,2,3\},$$
where $1,\bi,\bj,\bk$ are basis of $\bh_s$ satisfying the following equalities
\begin{equation}\label{rule}
\bi^2=-\bj^2=-\bk^2=-1, \bi\bj=\bk=-\bj\bi,\bj\bk=-\bi=-\bk\bj, \bk\bi=\bj=-\bi\bk.
\end{equation}
   Let $\bar{x}=x_0-x_1\bi-x_2\bj-x_3\bk$ be the conjugate of $x$ and
   \begin{equation}\label{Ix}I_x=\bar{x}x=x\bar{x}=x_0^2+x_1^2-x_2^2-x_3^2.\end{equation}
Obviously $\bc=\br\oplus\br\bi$, likewise  $\bh_s=\bc\oplus\bc\bj$ and $\bj z=\bar{z}\bj$ for $z\in \bc$.
 That is,  a split quaternion can be expressed  as
$$x=(x_0+x_1\bi)+(x_2+x_3\bi)\bj =z_1+z_2\bj=z_1+\bj\overline{z_2},z_1,z_2\in \bc.$$
A split quaternion is spacelike, timelike or lightlike  if $I_x<0, I_x > 0$ or $I_x = 0$, respectively.
 It can be easily verified that $$ \overline{xy}=\bar{y}\bar{x},\  I_{yx}=I_yI_x, \forall x,y\in\bh_s.$$
Let $$\Re(x)=(x+\bar{x})/2=x_0,\ \Im(x)=(x-\bar{x})/2=x_1\bi+x_2\bj+x_3\bk$$ be the real part and  imaginary part of $x$.
Then we have   $$\Re(xy)=\Re(yx)=x_0y_0-x_1y_1+x_2y_2+x_3y_3,\forall x,y\in \bh_s.$$

 Unlike the Hamilton quaternion algebra, the split quaternion algebra contains nontrivial zero divisors,
 nilpotent elements, and idempotents. For example, $\frac{1+\bj}{2}$ is an idempotent zero divisor, and $\bi-\bj$ is
 nilpotent.
  The set of zero divisors is denoted by  \begin{equation}
 Z(\bh_s)=\{x\in \bh_s:I_x=0\}.
 \end{equation}
 If $I_x\neq 0$ then $x$ is invertible and its inverse is  \begin{equation}x^{-1}=\frac{\bar{x}}{I_x}.\end{equation}
   If  $I_x=0$ then $x$ is not invertible.
    Cao and Chang [2]  defined its  Moore-Penrose inverse $x^+$ and  used it to solve some simple linear equations.

 For $x=x_0+x_1\bi+x_2\bj+x_3\bk,y=y_0+y_1\bi+y_2\bj+y_3\bk\in \bh_s$,
 we  define    \begin{equation}\label{inp2}\left\langle x,y \right\rangle=x_0y_0+x_1y_1-x_2y_2-x_3y_3.\end{equation}
For the sake of simplification with  little ambiguity, we also denote $$P_{xy}=P_{x,y}=\left\langle x,y \right\rangle.$$
Then we have
$$I_x=\left\langle x,x\right\rangle=P_{xx},\ \Re({\bar{y}x})=\Re({\bar{x}y})
=\Re({y\bar{x}})=\left\langle x,y \right\rangle=P_{xy}=P_{yx}.$$
Eq.(\ref{inp2}) can be thought of as the inner product of  the real vector 4-space  $\br^{2,2}$.
From this point of view, the split quaternion algebra is an algebraic presentation of $\br^{2,2}$.
 Similarly, we define  \begin{equation}\label{inp1} K_{xy}=K_{x,y}=-x_1y_1+x_2y_2+x_3y_3,\ \  M_x=K_{xx}.\end{equation}
 Then the imaginary part of split quaternions $\Im(\bh_s)$ with the inner product $K_{x,y}$
 can be thought of as the Minkowski  3-space  $\br^{2,1}$.
 Hence the  split quaternion algebra is closely related to these spaces $\br^{2,1}$,  $\br^{3,1}$ and $\br^{2,2}$ and  plays  an important   role in  modern  physics [1,5,8,10].

\subsection{Quadratic equations in $\br,\bc,\bh$}
 In algebra, a quadratic equation is any equation having the form
\begin{equation}\label{mainqeq}ax^{2}+bx+c=0,\end{equation}
where $x$ represents an unknown,  $a,b$ and $c$ represent known numbers and $a\neq 0$.

In real number setting, Eq.(\ref{mainqeq}) is solvable if and only if its discriminant $b^2-4ac\ge 0$.
 In complex number setting, by the fundamental theorem of algebra,  Eq.(\ref{mainqeq}) is always
 solvable,
  and its roots are expressed by the quadratic formula
\begin{equation}\label{rootforcom}x_{1,2}=\frac{-b\pm\sqrt{b^2-4ac}}{2a}.\end{equation}
By the well-known Vieta's formulas, the roots $x_{1},x_{2}$  satisfy
$$ x_{1}+x_{2}=\frac{-b}{a},\quad x_{1}x_{2}=\frac{c}{a}.$$

  The  quadratic equation  has been investigated  in  Hamilton quaternion  setting  in [6,9,12].
Huang and So [6] considered  $x^2+bx+c=0$ and  obtained explicit formulas of its roots.
 These formulas had been used in the classification of quaternionic M\"obius transformations [3,11].

\subsection{Quadratic equation in $\bh_s$}
It is interesting to consider the quadratic equation (\ref{mainqeq}) in split quaternions $\bh_s$.
In an algebra system, finding the roots of  the quadratic equation always connects with
 the factorizability of  quadratic polynomial [7].
 In $\br$ and $\bc$, the two problems are identical.
 In noncommutative algebra, the two problems are relevant.
  Scharler etc.[13] have considered  the factorizability of quadratic split quaternion polynomial.
  The result reveals some information of the roots of  split quaternion quadratic equation.
   Since the split quaternion algebra is noncommutative and contains nontrivial zero divisors,
    factorizability of quadratic split quaternion polynomial is quite different from finding its roots.

 In this paper, we will focus on deriving explicit formulas of the roots of  the quadratic equation $$ax^2+bx+c=0,a,b,c\in\bh_s.$$
   We believe  that   these formulas are important and  valuable as split quaternions is so important algebra
   in modern physics. However, this is a great challenge because there are  at least three difficulties to overcome:
  \begin{itemize}
    \item Too many parameters. There are three parameters $a,b,c\in \bh_s$, which amounts to twelve real parameters,
     in  $ax^2+bx+c=0$. Too many parameters make things complicated and difficult.
    \item Noncommutativity of split quaternions.
    \item Noninvertibility  of split quaternions in $Z(\bh_s)$.
  \end{itemize}

 We will use the following strategies to overcome these difficulties and find the roots of $ax^2+bx+c=0$.

\subsubsection{Reduce the number of  parameters}
    The first strategy we will use is to reduce the number of  parameters in $ax^2+bx+c=0$ to simplify our consideration.
    We will show that we only need to consider four types of quadratic equations.

      If $a$ is invertible, then  $ax^2+bx+c=0$ can be reformulated as $x^2+a^{-1}bx+a^{-1}c=0$.
      Therefore we only need to consider the following two types of quadratic equations:
  \begin{itemize}
    \item $x^2+bx+c=0;$
    \item $ax^2+bx+c=0, a\in Z(\bh_s)-\{0\}.$
  \end{itemize}

For the first type $x^2+bx+c=0$ with $b\notin \br$, we have the following proposition.
\begin{pro}\label{reducepro1} The quadratic equation $$y^2+dy+f=0,d=d_0+d_1\bi+d_2\bj+d_3\bj\notin \br$$ is solvable
 if and only if the quadratic equation $$x^2+bx+c=0,b=\Im(d)\neq 0,c=f-\frac{d_0}{2}(d-\frac{d_0}{2})$$ is solvable.
 If the quadratic equation $x^2+bx+c=0$ is solvable and $x$ is a solution then $y= x-\frac{d_0}{2}$ is a solution of
 $y^2+dy+f=0$.
\end{pro}

\begin{proof}
Rewriting
 $y^2+dy+f=0$ as
 $$(y+\frac{d_0}{2})^2+\Im(d)(y+\frac{d_0}{2})+f-\frac{d_0}{2}(d-\frac{d_0}{2})=0$$
 and letting $x=y+\frac{d_0}{2}$,  $b=\Im(d)$ and $c=f-\frac{d_0}{2}(d-\frac{d_0}{2})$, we prove  this proposition.
 \end{proof}
Hence for the first type we only need to solve the following equations:
 \begin{itemize}
    \item  Equation I: \quad  $x^2+bx+c=0, b\in \br;$
    \item  Equation II: \quad $x^2+bx+c=0, b=b_1\bi+b_2\bj+b_3\bk\neq 0.$
  \end{itemize}

For the second type, we have the following proposition.

\begin{pro}\label{reducepro2} The quadratic equation $dy^2+ey+f=0$ with  $d=d_1+d_2\bj\in Z(\bh_s)-\{0\}$,
$d_1,d_2\in \bc$ is solvable if and only if the quadratic equation $$ax^2+bx+c=0$$ is solvable, where
$$d_1^{-1}e=k_0+k_1\bi+k_2\bj+k_3\bk, k_i\in \br,i=0,\cdots,3$$ and
$$a=1+d_1^{-1}d_2\bj,\, b=d_1^{-1}e-k_0(1+d_1^{-1}d_2\bj),\,
c=d_1^{-1}f-\frac{d_1^{-1}ek_0}{2}+\frac{(1+d_1^{-1}d_2\bj)k_0^2}{4}.$$
If the quadratic equation $ax^2+bx+c=0$ is solvable and $x$ is a  solution then
$y= x-\frac{k_0}{2}$ is a solution of $dy^2+ey+f=0$.
\end{pro}
\begin{proof}
Since $d=d_1+d_2\bj\in Z(\bh_s)-\{0\}$, we have $I_{d_1}=I_{d_2}\neq 0$ and $d_1$ is invertible. Hence
 $dy^2+ey+f=0$ is equivalent to $$(1+d_1^{-1}d_2\bj)y^2+d_1^{-1}ey+d_1^{-1}f=0.$$ Let $y=x-\frac{k_0}{2}.$
 Then $dy^2+ey+f=0$ is equivalent to
 $$(1+d_1^{-1}d_2\bj)\big(x^2-k_0x+\frac{k_0^2}{4}\big)+d_1^{-1}e\big(x-\frac{k_0}{2}\big)+d_1^{-1}f=0.$$
 That is
  $$(1+d_1^{-1}d_2\bj)x^2+[d_1^{-1}e-k_0(1+d_1^{-1}d_2\bj)]x+d_1^{-1}f
  -\frac{d_1^{-1}ek_0}{2}+\frac{(1+d_1^{-1}d_2\bj)k_0^2}{4}=0.$$
 Let $$d_1^{-1}d_2=a_2+a_3\bi, a_2,a_3\in \br.$$ Then we have $a_2^2+a_3^2=1$ and
 $$a=1+d_1^{-1}d_2\bj=1+a_2\bj+a_3\bk\in Z(\bh_s).$$
 Since $\Re[d_1^{-1}e-k_0(1+d_1^{-1}d_2\bj)]=0$,  we have $$b=d_1^{-1}e-k_0(1+d_1^{-1}d_2\bj)=b_1\bi+b_2\bj+b_3\bk.$$
  \end{proof}
Hence for the second type we only need to solve the following equations:
 \begin{itemize}
    \item  Equation III: \quad  $ax^2+c=0, a=1+a_2\bj+a_3\bk\in Z(\bh_s)$;
    \item  Equation IV: \quad $ax^2+bx+c=0, a=1+a_2\bj+a_3\bk\in Z(\bh_s), b=b_1\bi+b_2\bj+b_3\bk\neq 0$.
  \end{itemize}

\subsubsection{Two real nonlinear systems}
The second strategy we will use is to reformulate $ax^2+bx+c=0$ as two real nonlinear systems.
Any solutions $x=x_0+x_1\bi+x_2\bj+x_3\bk$ of $ax^2+bx+c=0$ must fall into two categories:
 \begin{itemize}
  \item   $2x_0a+b\in Z(\bh_s)$;
    \item  $2x_0a+b\in \bh_s- Z(\bh_s)$.
     \end{itemize}

For Equations II and  IV, we define  that
$$SZ=\{x\in \bh_s:ax^2+bx+c=0 \mbox{ and }2x_0a+b\in Z(\bh_s)\}$$ and
$$SI=\{x\in \bh_s:ax^2+bx+c=0 \mbox{ and }2x_0a+b\in \bh_s-Z(\bh_s)\}.$$

Observe  that \begin{equation}x^2=x(2x_0-\bar{x})=2x_0x-I_x.\end{equation}
 Therefore  $ax^2+bx+c=0$  becomes \begin{equation}\label{linearx1} (2x_0a+b)x=aI_x-c.\end{equation}
Let
            \begin{eqnarray}\label{Nf1}N&=&I_x=\bar{x}x,\\
    \label{Tf1}T&=&\bar{x}+x=2x_0.\end{eqnarray}

If $2x_0a+b\in \bh_s-Z(\bh_s)$, then by (\ref{linearx1}) we have
\begin{equation}x=(2x_0a+b)^{-1}(aI_x-c)=(Ta+b)^{-1}(aN-c)=\frac{(T\bar{a}+\bar{b})(aN-c)}{T^2I_a+2TP_{ab}+I_b}\end{equation}
        and \begin{equation}\bar{x}=\frac{(N\bar{a}-\bar{c})(Ta+b)}{T^2I_a+2TP_{ab}+I_b}.\end{equation}
Substituting  the above formulas of $x$ and $\bar{x}$ in (\ref{Nf1}) and (\ref{Tf1}), we obtain
      \begin{eqnarray}
     \label{enf1} x\bar{x}&=&\frac{N^2I_a+I_c-2NP_{ac}}{T^2I_a+2TP_{ab}+I_b}=N,\\
     \label{enf2}
     x+\bar{x}&=&\frac{2TNI_a-T(\bar{a}c+\bar{c}a)+N(\bar{a}b+\bar{b}a)-(\bar{c}b+\bar{b}c)}{T^2I_a+2TP_{ab}+I_b}=T.
      \end{eqnarray}
Note that $$P_{xy}=\Re({\bar{y}x})=\frac{\bar{y}x+\bar{x}y}{2}.$$
Hence $(T,N)$ satisfies our first  real nonlinear  system:
      \begin{equation}\label{rsym1}\left\{
\begin{aligned}
    -I_aN^2+N(T^2I_a+2TP_{ab}+I_b+2P_{ac})-I_c=0,\\
     I_aT^3+2P_{ab}T^2+(2P_{ac}+I_b-2I_aN)T-2NP_{ab}+2P_{bc}=0.
  \end{aligned}
\right.\end{equation}

Since we aim to find a root of $ax^2+bx+c=0$, we do not know $x_0$ beforehand.
It is an embarrassing situation to assume that  $$2x_0a+b=Ta+b\in \bh_s-Z(\bh_s).$$
This embarrassing situation can be remedied as follows.
For Equations II and IV, we can solve the real  nonlinear system (\ref{rsym1}) to obtain the pair $(T,N)$.
 We can test whether or not $Ta+b\in \bh_s-Z(\bh_s)$. Only for the pair  $(T,N)$ such that $Ta+b\in \bh_s-Z(\bh_s)$,
  we obtain the corresponding  solution $ x=(Ta+b)^{-1}(aN-c).$

There may exist a solution $x=x_0+x_1\bi+x_2\bj+x_3\bk$ such that $2x_0a+b\in Z(\bh_s)$.
Such a situation is caused by the noninvertibility of split quaternions.
This property is a pivotal difference between  Hamilton quaternions and split quaternions.

If  $2x_0a+b\in Z(\bh_s)$ then
  \begin{equation}\label{ncasezero}\left\langle 2x_0a+b,2x_0a+b \right\rangle=4x_0^2I_a+4x_0P_{ab}+I_b=0.\end{equation}
Also we have
 \begin{equation}\label{nzacc}\left\langle aI_x-c,aI_x-c \right\rangle=I_aI_x^2-2I_xP_{ac}+I_c=0.\end{equation}
By Eq.(\ref{ncasezero}), we may know some information of $x_0$.
 For example, if  $P_{ab}\neq 0$ and $I_a=0$ then  $x_0=\frac{-I_b}{4P_{ab}}$.
However, such an information  is not enough to solve the quadratic equation.
  We will resort  to its natural real nonlinear system as followings.

Let $a=a_0+a_1\bi+a_2\bj+a_3\bk,$  $b=b_0+b_1\bi+b_2\bj+b_3\bk$,  $c=c_0+c_1\bi+c_2\bj+c_3\bk \in \bh_s$.
By the rule of multiplication (\ref{rule}), the  equation $ax^2+bx+c=0$ can be reformulated as our second real nonlinear
system:
\begin{equation}\label{rsym2}\left\{
\begin{aligned}
 &a_0(x_0^2-x_1^2+x_2^2+x_3^2)-2a_1x_0x_1+2a_2x_0x_2+2a_3x_0x_3+b_0x_0-b_1x_1+b_2x_2+b_3x_3+c_0=0,\\
  &             2a_0x_0x_1+a_1(x_0^2-x_1^2+x_2^2+x_3^2)-2a_2x_0x_3+2a_3x_0x_2+b_0x_1+b_1x_0-b_2x_3+b_3x_2+c_1=0,\\
 &2a_0x_0x_2-2a_1x_0x_3+a_2(x_0^2-x_1^2+x_2^2+x_3^2)+2a_3x_0x_1+b_0x_2-b_1x_3+b_2x_0+b_3x_1+c_2=0,\\
 &2a_0x_0x_3+2a_1x_0x_2-2a_2x_0x_1+a_3(x_0^2-x_1^2+x_2^2+x_3^2)+b_0x_3+b_1x_2-b_2x_1+b_3x_0+c_3=0.
\end{aligned}
\right.\end{equation}
Generally speaking,  Eqs.(\ref{rsym2}) is  a very  complicated real nonlinear system. It is  hard to solve it.
However, by our first strategy of simplification, we only need to consider several specific cases of $a,b,c$.

In fact, Equation I and Equation III are so special, by some properties of split quaternions,
we can solve them  directly by Eqs.(\ref{rsym2}).

For Equations II and IV, although $2x_0a+b\in Z(\bh_s)$ prevents us from using  the real system (\ref{rsym1}),
it compensates  us in simplifying real system (\ref{rsym2}).
In Equations II and IV, if the value of $x_0$ can be determined  by $2x_0a+b\in Z(\bh_s)$,
then we only need to find $x_1,x_2,x_3$ in Eqs.(\ref{rsym2}).
We can deduce  some linear relations of $x_1,x_2$ and $x_3$ form Eqs.(\ref{rsym2}), which can be used to  solve Eqs.(\ref{rsym2}), see pages 9 and 20 for more details.  If we know nothing about $x_0$ (for example, in case of  $I_a=0,P_{ab}=0$)
then we  can  also find some linear relation of $x_i,i=0,\cdots,3$, see Proposition 5.1. In these special cases,
 we can  obtain more relationships  of the coefficients $a,b,c$.
 These relationships  will help us to solve the real nonlinear system (\ref{rsym2}).

This strategy of using the above two real nonlinear systems help us to overcome the difficulty caused by the noncommutativity of split quaternions.
 Combining with the strategy of reducing the number of parameters, we also  partially overcome
  the difficulty  caused by noninvertibility  of split quaternions in $Z(\bh_s)$.
  Roughly speaking,  these strategies are our main tools in solving  $ax^2+bx+c=0$.

We list our  organization of this paper in the following table:

\begin{table}[!htbp]
	\centering
	\caption{The arrangement of this paper}
	\begin{tabular}{|c|c|c|c|}
		\hline
		Section &Type of Equation & Result&Examples of Theorem\\
		\hline
		2&Equation I& Theorem 2.1&Table 2 \\
 \hline
		\multirow{2}*{3}&Equation II for SZ&Theorem 3.1&Table 3\\
\cline{2-4}
 &Equation II for SI&Theorem 3.2&Tables 4 and 5\\
		\hline
4&Equation III& Theorem 4.1&Example 4.1\\
\hline
 		\multirow{6}*{5}&Equation IV with $P_{ab}\neq 0$ for SZ& Theorem 5.1&Examples 5.1 and 5.2\\
	\cline{2-4}
&Equation IV with $P_{ab}=0,b_1=a_2b_3-a_3b_2$ for SZ& Theorem 5.2&Examples 5.3 and 5.4\\
	\cline{2-4}
&Equation IV with $P_{ab}=0,b_1=a_3b_2-a_2b_3$ for SZ& Theorem 5.3&Examples 5.5 and 5.6\\
	\cline{2-4}
&Equation IV with $P_{ab}\neq 0$ for SI& Theorem 5.4&Example 5.7\\
	\cline{2-4}
&Equation IV with $P_{ab}=0,I_b+2P_{ac}\neq 0$ for SI& Theorem 5.5&Example 5.8\\
	\cline{2-4}
&Equation IV with $P_{ab}=0,I_b+2P_{ac}=0$ for SI& Theorem 5.6&Example 5.9\\
	\cline{2-4}
\hline
	\end{tabular}
\end{table}
In Table 1, for example,  Equation II for SZ means that   solving   Equation II for  $x\in SZ$,
the solvability conditions and the formulas of these solutions  are given in Theorem 3.1 and some examples
 are given in Table 3.
We remark that all examples in Table 1 are carefully chosen  to illustrate that all our formulas are work.

For Equation II and Equation IV, we need to consider their solutions in $SZ$ and $SI$, respectively.
 For this purpose, we choose the same quadratic equation in Table 3(2)I and Table 4(2)I, as well  Example 5.2 and Example
 5.7.
  The author has checked that all examples in Table 3 except Table 3 (2)I have no solution in $SI$;
 all examples in Tables 4 and  5 except Table 4 (2)I  do not satisfy Condition A, that is, these examples have no solution in
 $SZ$.
  We also have checked that Example 5.1 has no solution in $SI$.   In Theorem 5.2 and Theorem 5.3,
   we need implicit condition  $I_b=0$, while in Theorem 5.5 and Theorem 5.6,  we need implicit condition $I_b\neq 0$.
  These facts  mean that we have given all solutions of all examples in our paper.

\section{Equation I}
In this section, we consider Equation II. We begin with a definition.
\begin{defi}\label{def2.1}
	Let  $w=w_0+w_1\bi+w_2\bj+w_3\bk\in \bh_s$. We define that $\sqrt[s]{w}=\{x\in\bh_s:x^2=w\}.$
\end{defi}
By this definition, $\sqrt[s]{w}$ means the square root of $w$ in split quaternions.
We follow the conventional sign $\sqrt{z}$ for $z\in \br$ or $\bc$.
In [2,10], \"Ozdemir, Cao and Chang had obtained the root  of any split quaternions.
For our purpose, we rewrite these formulas of the square of  split quaternions  as follows.
\begin{pro}(cf.[2,10]) \label{prop2.1}
	Let  $w=w_0+w_1\bi+w_2\bj+w_3\bk\in \bh_s$.
	\begin{itemize}
		\item [(1)]  If $w\in \br$, that  is, $w=w_0$ then
		\begin{equation}\label{sqrtrn}\sqrt[s]{w_0}=\{x_1\bi+x_2\bj+x_3\bk:-x_1^2+x_2^2+x_3^2=w_0\}\
  \mbox{provided}\ w_0\le 0;\end{equation}
		\begin{equation}\label{sqrtrp}\sqrt[s]{w_0}=\{x_1\bi+x_2\bj+x_3\bk:
-x_1^2+x_2^2+x_3^2=w_0\}\cup \{\pm \sqrt{w_0}\}\
  \mbox{provided}\ w_0>0.\end{equation}
		
		\item [(2)] If $w\notin \br$ then $\sqrt[s]{w}\neq \emptyset$ if and only if $I_w\ge 0$ and $w_0+\sqrt{I_w}>0$.
		
		\begin{itemize}
			\item [(i)]  If $w_0-\sqrt{I_w}>0$ then
			\begin{equation}\label{sqrtnr2}\sqrt[s]{w}=\{\pm \frac{w+\sqrt{I_w}}{\sqrt{2(w_0+\sqrt{I_w})}}\}\cup \{\pm
\frac{w-\sqrt{I_w}}{\sqrt{2(w_0-\sqrt{I_w})}}\}.\end{equation}
			\item [(ii)] If $w_0+\sqrt{I_w}>0\ge w_0-\sqrt{I_w}$ then
			
\begin{equation}\label{sqrtnr1}\sqrt[s]{w}=\{\pm\frac{\sqrt{I_w}+w}{\sqrt{2(w_0+\sqrt{I_w})}}\}.\end{equation}
		\end{itemize}
	\end{itemize}
\end{pro}
\begin{proof}
 	Note that $x^2=x(2x_0-\bar{x})=x_0^2-x_1^2+x_2^2+x_3^2+2x_0(x_1\bi+x_2\bj+x_3\bk)$. If $x\in \sqrt{w}$ then
 	\begin{equation}\label{addeq1}2x_0x_1=w_1,2x_0x_2=w_2,2x_0x_3=w_3,x_0^2-x_1^2+x_2^2+x_3^2=w_0.\end{equation} 	
 Observe that $w\in \br$ implies that $x_0=0$ or $x_1^2+x_2^2+x_3^2=0$.
  This obsevation proves Eqs.(\ref{sqrtrn}) and (\ref{sqrtrp}). 	
  	By (\ref{addeq1}), if $w\notin \br$ then $x_0\neq 0$.
   Consequently,   $$x_1=\frac{w_1}{2x_0},x_2=\frac{w_2}{2x_0},x_3=\frac{w_3}{2x_0}$$ and
   \begin{equation}\label{eqx02}4x_0^4-4x_0^2w_0-w_1^2+w_2^2+w_3^2=0.\end{equation}
    	Viewing Eq.(\ref{eqx02}) as real quadratic equation with unknown $x_0^2$, we get its discriminant  $16I_w$.
     If $I_w\ge 0$ and $w_0+\sqrt{I_w}> 0$ then Eq.(\ref{eqx02}) is solvable. 	 	If $w_0-\sqrt{I_w}>0$ then
 $$x_0^2=\frac{w_0-\sqrt{I_w}}{2},\ \ \mbox{or}\  x_0^2=\frac{w_0+\sqrt{I_w}}{2}.$$
 If $w_0+\sqrt{I_w}>0\ge w_0-\sqrt{I_w}$ then $$x_0^2=\frac{w_0+\sqrt{I_w}}{2}.$$
 In each case we have $x=\frac{1}{2x_0}(2x_0^2+w_1\bi+w_2\bj+w_3\bk)$. This observation concludes the proof.
\end{proof}

We are ready to give our quadratic formulas for the case $b\in \br$.
\begin{thm}\label{thm2.1} Equation I is solvable if and only if
 $$\sqrt[s]{\frac{b^2-4c}{4}}\neq \emptyset.$$ If  Equation I is solvable, its solution(s) can be given by
	\begin{equation}\label{rootfor}x=\frac{-b}{2}+\sqrt[s]{\frac{b^2-4c}{4}}.\end{equation}
In other words, the solutions of Equation I can be obtained by
formulas according to the following cases.
\begin{itemize}
	\item [(1)]
 If $b,c\in \br$ and $b^2<4c$, then
$$x=\frac{-b}{2}+x_1\bi+x_2\bj+x_3\bk,$$
where $-x_1^2+x_2^2+x_3^2=\frac{b^2-4c}{4}$ and $x_1,x_2,x_3\in \br$.

\item [(2)]  If $b,c\in \br$ and $b^2\ge 4c$, then the set of solutions is
$$\{x=\frac{-b\pm \sqrt{b^2-4c}}{2}\}\cup\{x\in \bh_s: x=\frac{-b}{2}+x_1\bi+x_2\bj+x_3\bk\},$$
where $-x_1^2+x_2^2+x_3^2=\frac{b^2-4c}{4}$ and $x_1,x_2,x_3\in \br$.

\item [(3)]  If $b\in \br$, $c=c_0+c_1\bi+c_2\bj+c_3\bk\notin \br$, $(b^2-4c_0)^2-16K(c)\geq 0$
 and $\frac{b^2-4c_0+\sqrt{(b^2-4c_0)^2-16K(c)} }{2}>0$,  then
the solutions are as follows.
\begin{itemize}
 \item [(i)] If $\frac{b^2-4c_0-\sqrt{(b^2-4c_0)^2-16K(c)} }{2}> 0$ then
$$x=\frac{1}{2}(-b\pm \rho_i)\mp\frac{1}{\rho_i} (c_1\bi+c_2\bj+c_3\bk),i=1,2,$$
where $\rho_{1,2}=\sqrt{\frac{b^2-4c_0\pm \sqrt{(b^2-4c_0)^2-16K(c)} }{2}}$.

\item [(ii)] If $\frac{b^2-4c_0+\sqrt{(b^2-4c_0)^2-16K(c)} }{2}> 0\ge\frac{b^2-4c_0-\sqrt{(b^2-4c_0)^2-16K(c)} }{2}$,
    then
$$x=\frac{1}{2}(-b\pm \rho) \mp\frac{1}{\rho} (c_1\bi+c_2\bj+c_3\bk),$$
where $\rho=\sqrt{\frac{b^2-4c_0+\sqrt{(b^2-4c_0)^2-16K(c)} }{2}}$.
\end{itemize}
\end{itemize}
\end{thm}

\begin{proof}
Since $b\in \br$ we can rewrite $x^2+bx+c=0$ as
\begin{equation}\label{sqform}(x+\frac{b}{2})^2=\frac{b^2-4c}{4}.\end{equation}
If   $x^{2}+bx+c=0$ with $b\in \br$ is solvable then by  Definition \ref{def2.1} and   Eq.(\ref{sqform}),
we obtain $$x=\frac{-b}{2}+\sqrt[s]{\frac{b^2-4c}{4}}.$$
Expanding the part $\sqrt[s]{\frac{b^2-4c}{4}}$ by Proposition \ref{prop2.1}   concludes  the proof.
\end{proof}

\begin{exam}\label{exam2.1}
Consider the quadratic equation $x^2 +3+\bi+\bj+\bk=0$.  Since  $\sqrt[s]{-c}=\sqrt[s]{-3-\bi-\bj-\bk}=\emptyset$,
this quadratic equation is unsolvable.
\end{exam}
Some examples of Theorem \ref{thm2.1} are given in Table 1.
\begin{table}[!htbp]
	\centering
	\caption{Some examples in Theorem \ref{thm2.1}}
	\begin{tabular}{|c|c|l|}
		\hline
		& $(b, c)$ & The solution(s) of $x^2+bx+c=0$\\
		\hline
		(1)I&$(0,c_0)$,$(c_0>0)$ & $x=x_1\bi+x_2\bj+x_3\bk$ with  $-x_1^2+x_2^2+x_3^2=-c_0$ \\
 \hline
 \multirow{2}*{(1)II}&\multirow{2}*{$(0,c_0),\,(c_0<0)$} &$x=\pm \sqrt{-c_0}$\\ &&$x=x_1\bi+x_2\bj+x_3\bk$ with
 $-x_1^2+x_2^2+x_3^2=-c_0$
 \\		\hline
		\multirow{2}*{(2)}&\multirow{2}*{$(2,-3)$} &$x=1,x=-3$\\
		&&$x=-1+x_1\bi+x_2\bj+x_3\bk$ with $-x_1^2+x_2^2+x_3^2=4$ \\
		\hline
		\multirow{4}*{(3) (i)}&\multirow{4}*{$(4,\bi+2\bj+3\bk)$}&
  $x=-2+\sqrt{3}-\frac{\sqrt{3}}{6}\bi-\frac{\sqrt{3}}{3}\bj-\frac{\sqrt{3}}{2}\bk$\\
		&&$x=-2-\sqrt{3}+\frac{\sqrt{3}}{6}\bi+\frac{\sqrt{3}}{3}\bj+\frac{\sqrt{3}}{2}\bk$\\
		&&$x=-1-\frac{1}{2}\bi-\bj-\frac{3}{2}\bk$\\
		&&$x=-3+\frac{1}{2}\bi+\bj+\frac{3}{2}\bk$\\
		\hline
	    (3) (ii)&$(2,\frac{7}{2}\bi+\bj)$ & $x=\frac{1}{2}-\frac{7}{6}\bi-\frac{1}{3}\bj$ and
  $x=-\frac{5}{2}+\frac{7}{6}\bi+\frac{1}{3}\bj$\\
		\hline
	\end{tabular}
\end{table}

\section{Equation II}

\subsection{The solutions of form $2x_0+b\in Z(\bh_s)$}
In this subsection, we consider Equation II for SZ. We will find the necessary and sufficient conditions of
 Equation II having a solution $x\in 2x_0+b\in Z(\bh_s)$.

Suppose that $x^2+bx+c=0$ has a solution of $2x_0+b\in Z(\bh_s)$.
Since $a=1$ and  $b=b_1\bi+b_2\bj+b_3\bk\neq 0$, we have $$I_a=1,P_{ab}=0, P_{ac}=c_0, -I_b=M_b.$$
 By (\ref{ncasezero}) and (\ref{nzacc}), we have \begin{equation}x_0^2=\frac{M_b}{4}.\end{equation}
 and
 \begin{equation}\label{se3nzacc}I_x^2-2c_0I_x+I_c=0.\end{equation}
 The existence of $I_x$ leads to the discriminate  \begin{equation}\label{se3c}4c_0^2-4I_c=4M_c\geq 0.\end{equation}
  So we at first need that $$M_b,M_c\ge 0.$$

 Let $x=x_0+x_1\bi+x_2\bj+x_3\bk$  with $x_0^2=\frac{M_b}{4}$. For Equation II,
 the real system (\ref{rsym2}) can be reformulated as
\begin{eqnarray}
  -x_1^2+x_2^2+x_3^2-b_1x_1+b_2x_2+b_3x_3+x_0^2+c_0&=&0,\label{1geq1}\\
             2x_0x_1+b_3x_2-b_2x_3&=&-b_1x_0-c_1\label{1geq2},\\
             b_3x_1+2x_0x_2-b_1x_3&=&-b_2x_0-c_2\label{1geq3},\\
             -b_2x_1+b_1x_2+2x_0x_3&=&-b_3x_0-c_3\label{1geq4}.
 \end{eqnarray}
 Let  $y=(x_1,x_2,x_3)^T$.   Eqs.(\ref{1geq2})-(\ref{1geq4}) can be expressed as
  \begin{equation}\label{lineareq} Ay=u,\end{equation}
  where  \begin{equation}\label{matrixA1}A=\left(
  \begin{array}{ccc}
        2x_0 & b_3 & -b_2 \\
    b_3 & 2x_0 & -b_1 \\
   -b_2 & b_1 & 2x_0 \\
  \end{array}
\right),u=\left(
                 \begin{array}{c}
                   -b_1x_0-c_1 \\
                   -b_2x_0-c_2\\
                   -b_3x_0-c_3 \\
                 \end{array}
               \right).\end{equation}
      Note that $$\det(A)=8x_0^3+2x_0(b_1^2-b_2^2-b_3^2)=2x_0(4x_0^2-M_b)=0.$$ Let $$M=\left(
  \begin{array}{cc}
    2x_0 & -b_1 \\
    b_1 & 2x_0 \\
  \end{array}
\right).$$ Since $b=b_1\bi+b_2\bj+b_3\bk\neq 0$ and $M_b=-b_1^2+b_2^2+b_3^2\ge 0$,  the subdeterminant
  $$\det(M)=4x_0^2+b_1^2=b_2^2+b_3^2>0.$$  This means $rank(A)=2$.  We reformulate (\ref{1geq3}) and (\ref{1geq4}) as
 \begin{equation}\label{neq39}\left\{
\begin{aligned}
 	2x_0x_2-b_1x_3&=&-b_3x_1-b_2x_0-c_2,\\
 	b_1x_2+2x_0x_3&=&b_2x_1-b_3x_0-c_3.
 \end{aligned}
\right.\end{equation}
 Because $$M^{-1}=\frac{1}{b_2^2+b_3^2}\left(
\begin{array}{cc}
2x_0 & b_1 \\
-b_1 & 2x_0 \\
\end{array}
\right),$$
from Eqs. (\ref{neq39}) we have
\begin{equation}\label{eqx23}x_2=a_{21}+a_{22}x_1,\ x_3=a_{31}+a_{32}x_1,\end{equation}
where
\begin{equation}\label{eqa212}a_{21}=\frac{-2x_0^2b_2-(b_1b_3+2c_2)x_0-b_1c_3}{b_2^2+b_3^2},
\ a_{22}=\frac{b_1b_2-2b_3x_0}{b_2^2+b_3^2},\end{equation}
\begin{equation}\label{eqa312}a_{31}=\frac{-2x_0^2b_3+(b_1b_2-2c_3)x_0+b_1c_2}{b_2^2+b_3^2},
\ a_{32}=\frac{b_1b_3+2b_2x_0}{b_2^2+b_3^2}.\end{equation}
Let
\begin{equation}\label{se31C1}F_1=\frac{(b_2c_3-b_3c_2)x_0-b_1(b_2c_2+b_3c_3)+c_1(b_2^2+b_3^2)}{b_2^2+b_3^2}.\end{equation}
Substituting  $x_2,x_3$ of  Eqs.(\ref{eqx23}) in  Eq.(\ref{1geq2}), we must  have
    \begin{equation}\label{cBB}F_1=0.\end{equation}
In fact, since $\det(A)=0$, we must have that the coefficient of $x_1$ is zero.  And the condition $F_1=0$
 is just a restatement of the condition $rank(A)=rank(A,u)=2$.

Let \begin{equation}\label{se31C2}F_2=c_0(b_2^2+b_3^2)+b_1(b_3c_2-b_2c_3)+c_2^2+c_3^2.\end{equation}
Substituting $x_2,x_3$ of  Eqs.(\ref{eqx23}) in Eq.(\ref{1geq1}), we have
\begin{equation}\label{eqx1}2(b_3c_2-b_2c_3)x_1+F_2=0. \end{equation}
If $b_3c_2-b_2c_3=0$ we should have $F_2=0$ and in this case $x_1$ is arbitrary.
 If $b_3c_2-b_2c_3\neq 0$ then $$x_1=-\frac{F_2}{2(b_3c_2-b_2c_3)}.$$

Summarizing our reasoning process, we figure out the following  conditions.

  \begin{defi}\label{defcondA}
  For the coefficients $b,c$ in Equation II with   with $M_b,M_c\ge 0$,
    letting $r\in \br$ such that
    \begin{equation}\label{forr1}r^2=\frac{M_b}{4},\end{equation}
we say  $(b,c)$   satisfies  {\bf Condition A} if the following two conditions hold:
\begin{itemize}
		\item [(1)]  There exists an $r$ of (\ref{forr1}) satisfying
$$(b_2c_3-b_3c_2)r-b_1(b_2c_2+b_3c_3)+c_1(b_2^2+b_3^2)=0;$$
\item [(2)] If $b_3c_2=b_2c_3$ then $c_0(b_2^2+b_3^2)+c_2^2+c_3^2=0$.
\end{itemize}
 \end{defi}

Note that if $b_3c_2-b_2c_3\neq 0$ and the coefficients $b,c$ in Equation II satisfy Condition A,
then we have  $$r= \frac{b_1(b_2c_2+b_3c_3)-c_1(b_2^2+b_3^2)}{b_2c_3-b_3c_2}.$$
Summarizing the previous results, we obtain the following theorem.
\begin{thm}\label{thm3.1} Equation II has a  solution $x\in SZ$  if and only if Condition A holds.
 If Condition A is  hold by $r$, then we have the following cases.
\begin{itemize}
	\item [(1)]
 If $b_3c_2-b_2c_3\neq 0$  then Equation II has solutions
 $$x=r+x_1\bi+x_2\bj+x_3\bk,$$
where $$r= \frac{b_1(b_2c_2+b_3c_3)-c_1(b_2^2+b_3^2)}{b_2c_3-b_3c_2},$$
$$x_1= \frac{c_0(b_2^2+b_3^2)+b_1(b_3c_2-b_2c_3)+c_2^2+c_3^2}{2(b_2c_3-b_3c_2)}$$
and
 $$x_2=a_{21}+a_{22}x_1,x_3=a_{31}+a_{32}x_1,$$
where
 \begin{equation}\label{eqa212thm}a_{21}=\frac{-M_bb_2-(b_1b_3+2c_2)r-b_1c_3}{b_2^2+b_3^2},
 \ a_{22}=\frac{b_1b_2-2b_3r}{b_2^2+b_3^2};\end{equation}
\begin{equation}\label{eqa312thm}a_{31}=\frac{-M_bb_3+(b_1b_2-2c_3)r+b_1c_2}{b_2^2+b_3^2},
\ a_{32}=\frac{b_1b_3+2b_2r}{b_2^2+b_3^2}.\end{equation}

\item [(2)]  If $b_3c_2-b_2c_3=0$, then Equation II has solutions
$$x=r+x_1\bi+(a_{21}+a_{22}x_1)\bj+(a_{31}+a_{32}x_1)\bk,\forall x_1\in \br,$$
where  $a_{21},a_{22},a_{31},a_{32}$ are given by (\ref{eqa212thm}) and (\ref{eqa312thm})
 with $r= \frac{\sqrt{M_b}}{2}$ or $r=-\frac{\sqrt{M_b}}{2}$.
\end{itemize}
\end{thm}

Some examples of Theorem \ref{thm3.1} are given in Table 2.
\begin{table}[!htbp]
	\centering
	\caption{Some examples in Theorem \ref{thm3.1}}
	\begin{tabular}{|c|p{5.3cm}|c|p{7.6cm}|}
		\hline
		& $(b, c)$ &$r$& The solution(s) of $x^2+bx+c=0$\\
		\hline
		(1) &$(\bi+\bj,2+\bk)$ &$0$ & $x=\bi$ \\
		\hline
		(2)I &$(\bi+\bj,-1+\bi+\bj)$ &$0$ &  $x=x_1\bi+x_1\bj+\bk, \forall x_1\in \br$ \\
		\hline
		\multirow{2}*{(2)II}&\multirow{2}*{$(2\bi+2\bj+4\bk,-4+4\bi+4\bj+8\bk)$}&-2&$x=-2+x_1\bi+x_1\bj, \forall x_1\in
\br$\\
		\cline{3-4}
		&&2&$x=2+x_1\bi-(\frac{16}{5}+\frac{3x_1}{5})\bj+(\frac{-12}{5}+\frac{4x_1}{5})\bk, \forall x_1\in \br$ \\
		\hline
	\end{tabular}
\end{table}

\subsection{The solutions of form  $2x_0+b\in \bh_s-Z(\bh_s)$}
In this subsection we will consider Equation II for $SI$.

 Since $a=1$  and  $b=b_1\bi+b_2\bj+b_3\bk\neq 0$, we have $$I_a=1, P_{ab}=0.$$ Hence the real nonlinear system (\ref{rsym1})
 is simplified to
  \begin{eqnarray}
     \label{1enf1} N^2-(B+T^2)N+E=0,&\\
     \label{1enf2} T^3+(B-2N)T+D=0,
      \end{eqnarray}
      where $B=2P_{1,c}+I_b, E=I_{c},D=2P_{bc}$.

If $D\neq 0$ then by Eq.(\ref{1enf2}) we have  $T\neq 0$.
It follows from (\ref{1enf2}) that \begin{equation}\label{eqn3}N=\frac{T^3+BT+D}{2T}.\end{equation}
Substituting  the above in (\ref{1enf1}),  we get
$$\frac{(T^3+BT+D)^2}{4T^2}-\frac{(T^3+BT+D)[2T(B+T^2)]}{4T^2}+\frac{4T^2E}{4T^2}=0.$$
Hence we have
$$T^2(T^2+B)^2-4ET^2-D^2=0.$$
Let $T^2=z$. Then
\begin{equation}\label{cubeeq1}z^3+2Bz^2+(B^2-4E)z-D^2=0.\end{equation}
In order to find the pairs $(T,N)$ of Eqs.(\ref{1enf1}) and (\ref{1enf2}),
 we need to know all the positive solutions of Eq.(\ref{cubeeq1}) when $D\neq 0$ .

The following lemma provides  our required information  of positive solutions of Eq.(\ref{cubeeq1}).
\begin{lem}\label{lem1cub}
Let $B,E,D\in \br$ such that $D\neq 0$, $$F_1= \frac{-2B^3+2(B^2+12E)^{\frac{3}{2}}}{27}+\frac{8EB}{3}-D^2$$
 and
 $$F_2= \frac{-2B^3-2(B^2+12E)^{\frac{3}{2}}}{27}+\frac{8EB}{3}-D^2.$$
Then the cubic equation
\begin{equation}\label{cubeq}
z^3+2Bz^2+(B^2-4E)z-D^2=0 \end{equation}
has solutions in the interval $(0,\infty)$ as follows.

Case 1. If one of the following conditions holds, then  Eq.(\ref{cubeq}) has exactly one positive solution $z$.
\begin{itemize}
 \item [(i)]  $B^2+12E\le 0$;
  \item [(ii)] $B^2+12E>0,B\ge 0$;
  \item [(iii)] $B<0,B^2-4E<0$;
  \item [(iv)] $B^2+12E>0,B<0,B^2-4E>0,F_1<0$;
  \item [(v)] $B^2+12E>0,B<0,B^2-4E>0,F_2>0$.
  \end{itemize}

 Case 2. If one of the following conditions holds, then  Eq.(\ref{cubeq}) has exactly two positive solutions $z_1$ and $z_2$.
  \begin{itemize}
   \item [(vi)] $B^2+12E>0,B<0,B^2-4E>0,F_1=0$;
  \item [(vii)] $B^2+12E>0,B<0,B^2-4E>0,F_2=0$.
  \end{itemize}

   Case 3. If $B^2+12E>0,B<0,B^2-4E>0,F_1>0>F_2$, then Eq.(\ref{cubeq}) has exactly three positive solutions $z_i,i=1,2,3$.
\end{lem}

\begin{proof}
Let $f(z)=z^3+2Bz^2+(B^2-4E)z-D^2.$
We focus our attention on  the interval $(0,\infty)$.
 It is obvious that $$f(0)=-D^2<0,\ \lim_{z\to +\infty} f(z)=+\infty.$$
  Note that \begin{equation}f'(z)=3z^2+4Bz+B^2-4E=3\Big((z+\frac{2B}{3})^2-\frac{B^2+12E}{9}\Big).\end{equation}
The discriminant   of  $f'(z)=0$ is $$\Delta_{f'}=4(B^2+12E).$$  If $\Delta_{f'}\le 0$, then $f'(z)\geq 0$.
Hence Eq.(\ref{cubeq}) has exactly one positive solution $z$.
This proves Case 1 (i).

We now consider the case $\Delta_{f'}=4(B^2+12E)>0$.
Note that the solutions of $f'(z)=0$ are $$z_1=\frac{-2B-\sqrt{B^2+12E}}{3},z_2=\frac{-2B+\sqrt{B^2+12E}}{3}$$
and it can be verified that  $F_1=f(z_1),F_2=f(z_2)$.

If $B\ge 0$ then $z=-\frac{2B}{3}\leq 0$.
This means that $f'(z)$ is increasing in  $(-\frac{2B}{3},\infty)$.
If $z_2\le 0$ then $f'(z)>0$ in  $(0,\infty)$ and therefore $f(z)$ is increasing in $(0,\infty)$.
If $z_2>0$ then $f'(z)<0$ in $(0, z_2)$ and  $f'(z)>0$ in $(z_2,\infty)$.
Therefore  $f(z)$ is decreasing in $(0, z_2)$  then increasing in the interval $(z_2,\infty)$.
Note that $f(0)=-D^2<0$.  In both cases   Eq.(\ref{cubeq}) has exactly one positive solution $z$ in $(0,\infty)$.
This proves Case 1 (ii).

If  $B<0$  then $z=-\frac{2B}{3}>0$.  If $f'(0)=B^2-4E<0$ then  $z_1<0$ and  $f(z)$ decreases at first  in the interval
$(0,z_2)$
 and increases in the interval $(z_2,+\infty)$. Hence $f(z)$  has exactly one positive solution $z$. This proves Case 1
 (iii).

We now consider the case   $B<0$ and $f'(0)=B^2-4E>0$. In this case $f'(z)$ is positive in $(0,z_1)\cup (z_2,\infty)$ and
negative
 in $(z_1,z_2)$.

 If $f(z_1)<0$ then $f(z)$ is increasing in $(0,z_1)$, decreasing in $(z_1,z_2)$ and then increasing in $(z_2,+\infty)$.
 Hence Eq.(\ref{cubeq}) has exactly one positive solution in  $(z_2,+\infty)$. This proves Case 1 (iv).

 If $f(z_1)=0$,  then it is obvious that Eq.(\ref{cubeq}) has exactly two  positive solutions, the other one is in
 $(z_2,+\infty)$.
  This proves Case 2 (vi). 

  If $f(z_1)>f(z_2)>0$  then Eq.(\ref{cubeq}) has exactly one  positive solution in $[0,z_1]$. This proves Case 1 (v).

 If $f(z_1)>0=f(z_2)$  then Eq.(\ref{cubeq}) has exactly  two  positive solutions, the other one is  in $[0,z_1]$.
 This proves Case 2 (vii).
 If $f(z_1)>0>f(z_2)$  then Eq.(\ref{cubeq}) has exactly  three  positive solutions which lie  in
  $$(0,z_1)\cup (z_1,z_2)\cup (z_2,+\infty).$$ This proves Case 3.
 \end{proof}

We are ready to find the pairs $(T,N)$ of Eqs.(\ref{1enf1}) and (\ref{1enf2}).
\begin{lem}\label{lem2sys}
Let $B,E,D\in \br$.
Then the real system
\begin{eqnarray}
N^2-(B+T^2)N+E=0\label{eqn1},&\\
T^3+(B-2N)T+D=0\label{eqn2}
 \end{eqnarray}
has solutions $(T,N)\in \br^2$ as follows.
\begin{itemize}
  \item [(1)] $T=0,N=\frac{B\pm \sqrt{B^2-4E}}{2}$ provided $D=0,B^2-4E\ge 0$;
  \item [(2)] $T=\pm \sqrt{-2\sqrt{E}-B},N=-\sqrt{E}$; and $T=\pm \sqrt{2\sqrt{E}-B},N=\sqrt{E}$
  provided $D=0,E\ge 0,-2\sqrt{E}-B\ge 0$;
  \item [(3)] $T=\pm \sqrt{2\sqrt{E}-B},N=\sqrt{E}$ provided  $D=0,E\ge 0,2\sqrt{E}-B\ge 0>-2\sqrt{E}-B $;
    \item [(4)] $T=\pm \sqrt{z},N=\frac{T^3+BT+D}{2T}$ provided $D\neq 0$,
      Case 1 of Lemma \ref{lem1cub} holds and $z$ is the unique  positive root of real polynomial
      $z^3+2Bz^2+(B^2-4E)z-D^2$;
\item [(5)] $T=\pm \sqrt{z_i},N=\frac{T^3+BT+D}{2T},i=1,2$ provided $D\neq 0$,
 Case 2 of Lemma \ref{lem1cub} hold and $z_1,z_2$ are two  positive roots of real polynomial $z^3+2Bz^2+(B^2-4E)z-D^2$;
\item [(6)] $T=\pm \sqrt{z_i},N=\frac{T^3+BT+D}{2T},i=1,2,3$
  provided $D\neq 0$,  Case 3 of Lemma \ref{lem1cub} hold and $z_1,z_2,z_3$
  are the three positive roots of real polynomial $z^3+2Bz^2+(B^2-4E)z-D^2$.
\end{itemize}
\end{lem}

\begin{proof}
We divide our consideration into two subcases $D=0$ and $D\neq 0$.

We begin with  the subcase $D=0$. In this case Eq.(\ref{eqn2}) becomes
$$T(T^2+B-2N)=0.$$   Hence $T=0$ or $T^2+B-2N=0$.
If $T=0$, then Eq.(\ref{eqn1}) becomes $N^2-BN+E=0$ and therefore $N_{1,2}=\frac{B\pm \sqrt{B^2-4E}}{2}$ provided $B^2-4E\ge
0$.
 This proves (1).

If $T^2+B-2N=0$, then $T^2+B=2N$  and therefore Eq.(\ref{eqn1}) becomes  $N^2=E$. Thus $N_{1,2}=\pm \sqrt{E}$ provided $E\ge
0$.
Hence $T^2=2N-B=\pm 2\sqrt{E}-B$.  Take $N=-\sqrt{E},T^2=-2\sqrt{E}-B$ provided $E\ge 0,-2\sqrt{E}-B\ge 0$.
In this case we also can take $N=\sqrt{E},T^2=2\sqrt{E}-B$ because of  $2\sqrt{E}-B\ge -2\sqrt{E}-B\ge 0$.   This proves (2).

If $E\ge 0,2\sqrt{E}-B\ge 0>-2\sqrt{E}-B$ then we can take $N=\sqrt{E},T^2=2\sqrt{E}-B$. This proves (3).

\vspace{1mm}
For the second case $D\neq 0$, such a system can be solved by Lemma \ref{lem1cub} as claimed. These prove (4),(5) and (6).
\end{proof}

\begin{thm}\label{thm3.2}  For the coefficients $b,c$ in Equation II, we define
 \begin{equation}\label{coeff}B=2P_{1,c}+I_b, E=I_{c},D=2P_{bc}.\end{equation} If Equation II is solvable
  and has solution  $x=x_0+x_1\bi+x_2\bj+x_3\bk$ with $x_0^2\neq \frac{M_b}{4}$, then
\begin{equation}x=(T+b)^{-1}(N-c),\end{equation}

where  $(T,N)$ is chosen as follows.

\begin{itemize}
  \item [(1)] $T=0,N=\frac{B\pm \sqrt{B^2-4E}}{2}$ provided $D=0,B^2-4E\ge 0$;
  \item [(2)] $T=\pm \sqrt{-2\sqrt{E}-B},N=-\sqrt{E}$, and $T=\pm \sqrt{2\sqrt{E}-B},N=\sqrt{E}$
   provided $D=0,E\ge 0,-2\sqrt{E}-B\ge 0$;
  \item [(3)] $T=\pm \sqrt{2\sqrt{E}-B},N=\sqrt{E}$ provided  $D=0,E>0,2\sqrt{E}-B\ge 0 >-2\sqrt{E}-B $;

  \item [(4)] $T=\pm \sqrt{z},N=\frac{T^3+BT+D}{2T}$ provided $D\neq 0$,  Case 1 of Lemma \ref{lem1cub} holds and $z$
  is the unique  positive root of real polynomial $z^3+2Bz^2+(B^2-4E)z-D^2$;
\item [(5)] $T=\pm \sqrt{z_i},N=\frac{T^3+BT+D}{2T},i=1,2$ provided $D\neq 0$,
 Case 2 of Lemma \ref{lem1cub} holds and $z_1,z_2$ are two  positive roots of real polynomial $z^3+2Bz^2+(B^2-4E)z-D^2$;
\item [(6)] $T=\pm \sqrt{z_i},N=\frac{T^3+BT+D}{2T},i=1,2,3$  provided $D\neq 0$,
 Case 3 of Lemma \ref{lem1cub} holds and $z_1,z_2,z_3$ are the three positive roots of real polynomial
 $z^3+2Bz^2+(B^2-4E)z-D^2$.
\end{itemize}
\end{thm}

\begin{proof}
By  Lemma \ref{lem2sys},  we get  the pairs $(T,N)$ of Eqs. (\ref{eqn1}) and (\ref{eqn2}).
   For each pair $(T,N)$, if  $T+b$  is invertible then we get a solution  $x=(T+b)^{-1}(N-c)$.
\end{proof}
Some examples of Theorem  \ref{thm3.2} are given in Tables 4 and 5.
  In Table 5,  "C1(i)L3.1"  is an abbreviation of Case 1 (i) of Lemma \ref{lem1cub}.

\begin{table}[!htbp]
	\centering
	\caption{Some examples in Theorem \ref{thm3.2} with $D=0$}
	\begin{tabular}{|c|c|c|p{5.8cm}|}
		\hline
		& $(b, c)$ & $(T, N)$ & The solution(s) of $x^2+bx+c=0$\\
		\hline
		\multirow{2}*{(1)} &\multirow{2}*{$(2\bj,2\bi+3\bk)$} &$(0,1)$ & $x=\frac{3}{2}\bi+\frac{1}{2}\bj+\bk$ \\
		\cline{3-4}
		&&$(0,-5)$ & $x=\frac{3}{2}\bi-\frac{5}{2}\bj+\bk$\\
		\hline
		\multirow{3}*{(2)I} & \multirow{3}*{$(\bi+\bj,-1+\bi+\bj)$} &$(0,0),(0,-1),(0,-2)$ & no solution \\
		\cline{3-4}
		&&$(-2,1)$ & $x=-1$ \\
		\cline{3-4}
		&& $(2,1)$ & $x=1-\bi-\bj$ \\
		\hline
		\multirow{4}*{(2)II} & \multirow{4}*{$(2\bj+\bk,2\bi)$} &$(1,-2)$ &
$x=\frac{1}{2}+\frac{1}{2}\bi-\frac{3}{2}\bj+\frac{1}{2}\bk$ \\
		\cline{3-4}
		&&$(-1,-2)$ & $x=-\frac{1}{2}-\frac{1}{2}\bi-\frac{3}{2}\bj+\frac{1}{2}\bk$ \\
		\cline{3-4}
		&& $(3,2)$ & $x=\frac{3}{2}-\frac{3}{2}\bi-\frac{1}{2}\bj-\frac{3}{2}\bk$ \\
		\cline{3-4}
		&& $(-3,2)$ & $x=-\frac{3}{2}+\frac{3}{2}\bi-\frac{1}{2}\bj-\frac{3}{2}\bk$ \\
		\hline
		\multirow{2}*{(3)} & \multirow{2}*{$(\bk,5\bi+3\bj)$} &$( 3,4)$ &
$x=\frac{3}{2}-\frac{3}{2}\bi-\frac{1}{2}\bj-\frac{1}{2}\bk$ \\
		\cline{3-4}
		&&$( -3,4)$ &  $x=-\frac{3}{2}+\frac{9}{4}\bi+\frac{7}{4}\bj-\frac{1}{2}\bk$ \\
		\hline
	\end{tabular}
\end{table}

\begin{table}[!htbp]
	\centering
	\caption{Some examples in Theorem \ref{thm3.2} with $D\neq 0$}
	\begin{tabular}{|p{2.75cm}|p{3.6cm}|p{2.78cm}|p{5.1cm}|}
		\hline
		& \centering $(b, c)$ & \centering $(z,T, N)$ &solution(s) of $x^2+bx+c=0$\\
		\hline
		\multirow{2}*{(4) $\&$ C1(i)L3.1}& & $(\frac{1}{2},\frac{\sqrt{2}}{2},\frac{1-\sqrt{2}}{4})$
 &$x=\frac{\sqrt{2}}{4}+\frac{\sqrt{2}-1}{2}\bi+\frac{\sqrt{2}-2}{4}\bj+\frac{1}{2}\bk$\\
		\cline{3-4}
		&\centering $(\bi+\bj,\frac{\bj}{4})$&$(\frac{1}{2},-\frac{\sqrt{2}}{2},\frac{1+\sqrt{2}}{4})$
 &$x=-\frac{\sqrt{2}}{4}-\frac{\sqrt{2}+1}{2}\bi-\frac{\sqrt{2}+2}{4}\bj+\frac{1}{2}\bk$\\
		\hline

		\multirow{2}*{(4) $\&$   C1(ii)L3.1}&& $(4,2,4)$ &$x=1-2\bi-\frac{4}{5}\bj-\frac{3}{5}\bk$\\
		\cline{3-4}
		&\centering $(\bi,3\bi+\bj+2\bk)$ &$(4,-2,1)$&$x=-1+\bi+\bk$\\
		\hline

		\multirow{2}*{(4) $\&$   C1(iii)L3.1}& & $(4,2,2)$ &$x=1-3\bi+2\bj-2\bk$\\
		\cline{3-4}
		&\centering $(\bi+2\bk,1+\bi)$ &$(4,-2,1)$ &$x=-1+2\bi+2\bj$\\
		\hline

		\multirow{2}*{(4) $\&$ C1(iv)L3.1}& & $(4,2,2)$ &$x=1-\frac{9}{4}\bi-\frac{1}{4}\bj-2\bk$\\
		\cline{3-4}
		&\centering $(\frac{11}{4}\bi+\frac{3}{4}\bj+3\bk,\bi+\bj)$&$(4,-2,0)$
&$x=-1+\frac{5}{2}\bi+\frac{5}{2}\bj+\bk$\\
			\hline
		
		\multirow{2}*{(4) $\&$ C1(v)L3.1}&& $(1,1,-\sqrt{11}-3)$
 &$x=\frac{1}{2}-\frac{2+\sqrt{11}}{2}\bi+\frac{1}{2}\bj-\frac{\sqrt{11}+4}{2}\bk$\\
		\cline{3-4}
		&\centering $(2\bi+\sqrt{11}\bk,\bj+\bk)$ &$(1,1,\sqrt{11}-3)$
&$x=-\frac{1}{2}+\frac{\sqrt{11}-6}{6}\bi+\frac{1}{6}\bj+\frac{8-3\sqrt{11}}{6}\bk$\\
		\hline
		\multirow{4}*{(5) $\&$ C2(vi)L3.1}& & $(1,1,0)$ &$x=\frac{1}{2}-\frac{3}{2}\bi+\frac{3}{2}\bj+\frac{1}{2}\bk$\\
		\cline{3-4}
		&\centering $(-\bi+2\bk,-\bi+\bj)$&$(1,-1,-2)$ &$x=-\frac{1}{2}+\frac{1}{2}\bi+\frac{1}{2}\bj-\frac{3}{2}\bk$\\
		\cline{3-4}
		&&$(4,2,1)$ &$x=1+5\bi-4\bj-3\bk$\\
		\cline{3-4}
		&&$(4,-2,0)$ &$x=-1-\bk$\\
		\hline

		\multirow{4}*{(5) $\&$ C2(vii)L3.1}&& $(4,2,-5)$ &$x=1-\frac{1}{2}\bi-\frac{3}{2}\bj-2\bk$\\
		\cline{3-4}
		&\centering $(2\bj+2\bk,2\bj+\bk)$ &$(4,-2,1)$ &$x=-1-\frac{1}{2}\bi-\frac{1}{2}\bj$\\
		\cline{3-4}
		&&$(6,\sqrt{6},-1-\sqrt{6})$ &$x=\frac{\sqrt{6}}{2}-\bi-\bj-(\frac{\sqrt{6}}{2}+1)\bk$\\
		\cline{3-4}
		&&$(6,-\sqrt{6},\sqrt{6}-1)$ &$x=-\frac{\sqrt{6}}{2}-\bi-\bj+(\frac{\sqrt{6}}{2}-1)\bk$\\
		\hline

		\multirow{6}*{(6) $\&$   C3L3.1}&& $(1,1,-6)$ &$x=\frac{1}{2}-\frac{5}{6}\bi-\frac{3}{2}\bj-\frac{13}{6}\bk$\\
		\cline{3-4}
		& & $(1,-1,0)$ &$x=-\frac{1}{2}-\bi-\bj-\frac{1}{2}\bk$\\
		\cline{3-4}
		&\centering $(2\bj+3\bk,3+\bi+3\bj-\bk)$&$(4,2,-3)$ &$x=1-\bi-\bj-2\bk$\\
		\cline{3-4}
		&&$(4,-2,0)$ &$x=-1-\frac{13}{9}\bi-\frac{15}{9}\bj-\frac{5}{9}\bk$\\
		\cline{3-4}
		&&$(9,3,0)$ &$x=\frac{3}{2}-2\bi-\frac{5}{2}\bk$\\
		\cline{3-4}
		&&$(9,-3,2)$ &$x=-\frac{3}{2}-\frac{7}{2}\bi-\frac{7}{2}\bj+\frac{1}{2}\bk$\\
		\hline	
	\end{tabular}
\end{table}

\subsection{Some technical details}
Table 4 (2)I shows that we may have pairs (T,N) such that $T+b\in Z(\bh_s)$.
 In such case we can not use the formula $x=(T+b)^{-1}(N-c)$.
 By  the cubic equation (\ref{cubeq}), we can figure out this
situation  in advance.

\begin{pro}\label{mainprop1}  
Let $b=b_1\bi+b_2\bj+b_3\bk\neq 0 ,c=c_0+c_1\bi+c_2\bj+c_3\bk\in\bh_s$ with $M_b\ge 0$
 and  \begin{equation}\label{coeffadd1}B=2P_{1,c}-M_b, E=I_{c},D=2P_{bc}.\end{equation}
	Let $T$ be given by Lemma \ref{lem1cub}.
Then there exist  a  $T$  such that $T+b\in Z(\bh_s)$  if and only if
\begin{equation}\label{condD}M_bM_c=K_{bc}^2.\end{equation}
\end{pro}

\begin{proof} It is obvious that $T+b\in Z(\bh_s)$ implies $T^2=M_b$.  If $z=M_b$ is a root of $$z^3+2Bz^2+(B^2-4E)z-D^2=0$$ then
\begin{equation}M_b\big(M_b^2+2BM_b+B^2-4E\big)-D^2=0.\end{equation}
Because  $$M_b+B=2c_0,4c_0^2-4E=4M_c,D^2=4P^2_{bc}=4K^2_{bc},$$  we have
$$M_bM_c=K_{bc}^2.$$
\end{proof}

The pairs of  $(b,c)$ of Tables 4 and 5  except Table 4 (2)I do not satisfy  Eq.(\ref{condD}).
To get all solutions of Equation II, we need use Theorems \ref{thm3.1} and \ref{thm3.2} together. For example,
 from Table 3 (2)I and Table 4 (2)I, we know that the set of solutions of  the equation $$x^2 +(\bi+\bj)x-1+\bi+\bj= 0$$
is $$\{-1,1-\bi-\bj\}\cup \{x_1\bi+x_1\bj+\bk, \forall x_1\in \br\}.$$

\section{Equation III}

For later use and reduce duplication, we begin with  the quadratic equation $$ax^2+bx+c=0,a=1+a_2\bj+a_3\bk\in
Z(\bh_s),b=b_1\bi+b_2\bj+b_3\bk.$$
The real nonlinear system (\ref{rsym2}) of the above equation  reduces to
\begin{eqnarray}
 x_0^2-x_1^2+x_2^2+x_3^2+2a_2x_0x_2+2a_3x_0x_3-b_1x_1+b_2x_2+b_3x_3+c_0=0,\label{ae1}\\
               2x_0x_1-2a_2x_0x_3+2a_3x_0x_2+b_1x_0-b_2x_3+b_3x_2+c_1=0,\label{ae2}\\
 2x_0x_2+a_2(x_0^2-x_1^2+x_2^2+x_3^2)+2a_3x_0x_1-b_1x_3+b_2x_0+b_3x_1+c_2=0,\label{ae3}\\
 2x_0x_3-2a_2x_0x_1+a_3(x_0^2-x_1^2+x_2^2+x_3^2)+b_1x_2-b_2x_1+b_3x_0+c_3=0.\label{ae4}
\end{eqnarray}

\subsection{Equation $ax^2+c=0$ with $a=1+a_2\bj+a_3\bk\in Z(\bh_s)$}

 In this subsection, we consider $ax^2+c=0$ with $a=1+a_2\bj+a_3\bk\in Z(\bh_s)$.

  We begin with the definition of  Moore-Penrose inverse   in split quaternions \cite{cao}.
  The  Moore-Penrose inverse of  $a=t_1+t_2\bj,t_1,t_2\in \bc$ is defined to be
$$a^+=\left\{\begin{array}{ll}
0, & \hbox{if\, a=0;} \\
\frac{\overline{t_1}-t_2\bj}{|t_1|^2-|t_2|^2}=\frac{\overline{a}}{I_a}, & \hbox{if\, $I_a\neq 0$;} \\
\frac{\overline{t_1}+t_2\bj}{4|t_1|^2}, & \hbox{if\, $I_a=0$.} \\
\end{array}
\right.$$
For $a=t_1+t_2\bj\in  Z(\bh_s)-\{0\}$, we have the following equations:
 \begin{equation*}aa^+a=a,\  a^+aa^+=a^+,\  aa^+=\frac{1}{2}\big(1+\frac{t_2}{\overline{t_1}}\bj\big),
 \  a^+a=\frac{1}{2}\big(1+\frac{t_2}{t_1}\bj\big).\end{equation*}

\begin{lem}[cf.\cite{cao}]\label{lemmoor}
	Let $a=t_1+t_2\bj\in  Z(\bh_s)-\{0\}$. Then the equation $ax=d$ is  solvable
	if and only if $aa^+d=\frac{1}{2}(1+\frac{t_2}{\overline{t_1}}\bj)d=d$,
 in which case all solutions are given by $$x=a^+d+(1-a^+a)y
 =\frac{\overline{t_1}+t_2\bj}{4|t_1|^2}d+\frac{1}{2}(1-\frac{t_2}{t_1}\bj)y,\forall y\in \bh_s.$$
\end{lem}

 \begin{thm}\label{thm4.1}
  The quadratic equation  $ax^2+c=0$ is solvable if and only if $$ac=2c.$$
  \begin{itemize}
		\item [(1)] If  $c_1=0$ then $ax^2+c=0$ has solutions
  \begin{equation}x=x_1\bi+x_2\bj+x_3\bk,\end{equation}
   where $-x_1^2+x_2^2+x_3^2+c_0=0.$  Moreover, if $c_0\leq 0$ then $ax^2+c=0$ also has
 solutions \begin{equation}x=\pm \sqrt{-c_0}-a_2x_2-a_3x_3+(a_2x_3-a_3x_2)\bi+x_2\bj+x_3\bk,\forall x_2,x_3\in
 \br.\end{equation}

  \item [(2)] If $c_1 \neq 0$ then $ax^2+c=0$ has  solutions
  $$ x=T+x_1\bi+x_2\bj+x_3\bk, \forall x_2,x_3\in\br,$$
where $T\in \br$ is a solution of the following quartic equation
\begin{eqnarray}
z^4+2(a_2x_2+a_3x_3)z^3+[(a_2x_2+a_3x_3)^2+c_0]z^2
+(a_2x_3-a_3x_2)c_1z-\frac{c_1^2}{4}=0\label{sx0thm}
 \end{eqnarray}
and
 $$x_1=\frac{-c_1}{2T}+a_2x_3-a_3x_2.$$
  \end{itemize}
   \end{thm}

\begin{proof}
Because  $a=1+(a_2+a_3\bi)\bj\in Z(\bh_s)$, we have
  \begin{equation}\label{ainvf}a^+=\frac{a}{4},\ \ aa^+=a^+a=\frac{a}{2}.\end{equation}
  If $ax^2+c=0$ is solvable, we assume $x$ is a solution of it.  Let $Y=x^2$. Then
$$aY=-c.$$
By Lemma \ref{lemmoor}, the above equation has a solution $Y$ if and only if
$$aa^+c=c.$$
That is  $$ac=2c.$$
The condition $ac=2c$ can be reformulated as  the following four equations:
$$-c_0+a_2c_2+a_3c_3=0, -c_1-a_2c_3+a_3c_2=0,$$
$$ -c_2+a_2c_0+a_3c_1=0, -c_3-a_2c_1+a_3c_0=0.$$
Under the conditions $ac=2c$ and $b=0$, we have
\begin{center}Eq.(\ref{ae3})$=$ Eq.(\ref{ae1})$\times a_3-$Eq.(\ref{ae2})$\times a_2$,\quad
  Eq.(\ref{ae4})$=$ Eq.(\ref{ae1})$\times a_2-$Eq.(\ref{ae2})$\times a_3$.\end{center}
  Therefore,  $ax^2+c=0$  with $ac=2c$  actually only has the following  two independent equations in (\ref{rsym2}):
\begin{eqnarray}
 x_0^2-x_1^2+x_2^2+x_3^2+2(a_2x_2+a_3x_3)x_0+c_0=0,\label{beq1s4}\\
               2x_0(x_1-a_2x_3+a_3x_2)+c_1=0.\label{beq2s4}
 \end{eqnarray}
To solve the above real system, we divide it into two cases: $$c_1=0 \mbox{ and } c_1\neq 0.$$

If $c_1=0$ then by Eq.(\ref{beq2s4}) we have $$x_0=0 \mbox{ or  }x_1-a_2x_3+a_3x_2=0.$$
If $x_0=0$ holds, substituting it in Eq.(\ref{beq1s4}), we get $-x_1^2+x_2^2+x_3^2+c_0=0$.
So we have a solution $$x=x_1\bi+x_2\bj+x_3\bk,\mbox{ where }-x_1^2+x_2^2+x_3^2+c_0=0.$$
If  $x_1-a_2x_3+a_3x_2=0$ then we have $x_1=a_2x_3-a_3x_2$.
Substituting this in Eq.(\ref{beq1s4}), we obtain
$$(x_0+a_2x_2+a_3x_3)^2+c_0=0.$$
If $c_0\leq 0$ then $x_0+a_2x_2+a_3x_3=\pm \sqrt{-c_0}$.
And therefore $x_0=\pm \sqrt{-c_0}-a_2x_2-a_3x_3$.
Hence we have a solution $$x=\pm \sqrt{-c_0}-a_2x_2-a_3x_3+(a_2x_3-a_3x_2)\bi+x_2\bj+x_3\bk,\forall x_2,x_3\in \br.$$

 If $c_1\neq 0$ then by  Eq.(\ref{beq2s4}) we have  $x_0\neq 0$ and  $x_1-a_2x_3+a_3x_2\neq 0$.
So we have
$$x_1=\frac{-c_1}{2x_0}+a_2x_3-a_3x_2.$$
Substituting this in Eq.(\ref{beq1s4}), we obtain
\begin{eqnarray}
x_0^4+2(a_2x_2+a_3x_3)x_0^3+[(a_2x_2+a_3x_3)^2+c_0]x_0^2
+(a_2x_3-a_3x_2)c_1x_0-\frac{c_1^2}{4}=0.\label{sx0}
 \end{eqnarray}
Let \begin{eqnarray}
f(z)=z^4+2(a_2x_2+a_3x_3)z^3+[(a_2x_2+a_3x_3)^2+c_0]z^2
+(a_2x_3-a_3x_2)c_1z-\frac{c_1^2}{4}.\label{sx0z}
 \end{eqnarray}
 Then $$f(0)=-\frac{c_1^2}{4}<0,\lim_{z\to +\infty} f(z)=+\infty,\lim_{z\to -\infty} f(z)=+\infty.$$
 This means that $f(z)=0$ has at least two real  solutions  $z_1\in (-\infty,0)$ and $z_2\in (0,\infty)$ for any $x_2,x_3\in
 \br$.
Let $T$ be a real solution of $f(z)=0$. Then Equation III has solutions
$$ x=T+x_1\bi+x_2\bj+x_3\bk,\forall x_2,x_3\in \br,$$
where
$$x_1=\frac{-c_1}{2T}+a_2x_3-a_3x_2.$$
\end{proof}

\begin{exam}\label{exam4.1}
Consider the quadratic equation $(1+\bj)x^2 -1-\bj=0$.
That is, $a=1+\bj,c=-1-\bj.$
 Hence we obtain two solutions:
 $$x=x_1\bi+x_2\bj+x_3\bk,$$ where $-x_1^2+x_2^2+x_3^2-1=0.$
 Since $c_0=-1<0$, we also have
 solutions $$x=(\pm 1-x_2)+x_3\bi+x_2\bj+x_3\bk,\forall x_2,x_3\in \br.$$
\end{exam}

\begin{exam}\label{exam4.2}
Consider the quadratic equation $(1+\frac{\sqrt{2}}{2}\bj+\frac{\sqrt{2}}{2}\bk)x^2 +2+\bi+\frac{3\sqrt{2}}{2}\bj
+\frac{\sqrt{2}}{2}\bk= 0$.
That is, $a=1+\frac{\sqrt{2}}{2}\bj+\frac{\sqrt{2}}{2}\bk,c=2+\bi+\frac{3\sqrt{2}}{2}\bj+\frac{\sqrt{2}}{2}\bk.$
Since $c_1=1$, we have solutions $$x=T+x_1\bi+x_2\bj+x_3\bk, \forall x_2,x_3\in\br,$$
where $T\in \br$ is a solution of the following quartic equation
\begin{eqnarray}
z^4+\sqrt{2}(x_2+x_3)z^3+[\frac{1}{2}(x_2+x_3)^2+2]z^2
+\frac{\sqrt{2}}{2}(x_3-x_2)z-\frac{1}{4}=0\label{sx0thmexa}
 \end{eqnarray}
and
 $$x_1=\frac{-c_1}{2T}+a_2x_3-a_3x_2.$$
For example, if we set $x_2=2,x_3=3$, then  Eq.(\ref{sx0thmexa}) has two solutions
 $T_1=-0.1658$ and $T_2=0.1069$.
  So we have  two solutions
 $$x=0.1069-3.9708\bi+2\bj+3\bk \mbox{ and } x=-0.1658+3.7222\bi+2\bj+3\bk.$$
\end{exam}

\section{Equation IV}

In this section, we will consider Equation IV, that is, the quadratic equation
$$ax^2+bx+c=0,a=1+a_2\bj+a_3\bk\in Z(\bh_s),b=b_1\bi+b_2\bj+b_3\bk\neq 0.$$
For the sake of simplification, we define the following three numbers of Equation IV:
 \begin{equation}\label{t1t2d}\delta=a_2b_3-a_3b_2+b_1,t_1=c_2-c_0a_2-a_3c_1,t_2=c_3-c_0a_3+a_2c_1.\end{equation}

We begin with a proposition for later use. The proposition describe the linear relation of $x_i,i=0,\cdots, 3$.
\begin{pro}\label{prop5.1}
Suppose that $P_{ab}=0$ and $I_a=0$. Then the solution $x$ of $ax^2+bx+c=0$ satisfies the following linear equation:
  \begin{equation}\label{lineareqn} Ay=u,\end{equation}
  where  $y=(x_0,x_1,x_2,x_3)^T$ and \begin{equation}\label{matrixA1n}A=\left(
  \begin{array}{cccc}
      b_2-a_3b_1&a_2b_1+b_3&0&-\delta\\
   a_2b_1+b_3&a_3b_1-b_2&\delta&0
    \end{array}
\right),u=\left(
                 \begin{array}{c}
                   -t_1\\
                   -t_2
                   \end{array}
               \right).\end{equation}
\end{pro}
\begin{proof}
Note that $a_2^2+a_3^2=1$ and $a_2b_2+a_3b_3=0$.
 Using Eq.(\ref{ae3})$-$Eq.(\ref{ae1})$\times a_2-$Eq.(\ref{ae2})$\times a_3$, we have
\begin{equation}\label{new3}(b_2-a_3b_1)x_0+(a_2b_1+b_3)x_1+(a_3b_2-a_2b_3-b_1)x_3+c_2-a_2c_0-a_3c_1=0.\end{equation}
Using Eq.(\ref{ae4})$-$Eq.(\ref{ae1})$\times a_3+$ Eq.(\ref{ae2})$\times a_2$, we have
\begin{equation}\label{new4}(a_2b_1+b_3)x_0+(a_3b_1-b_2)x_1+(a_2b_3-a_3b_2+b_1)x_2+c_3+a_2c_1-a_3c_0=0.\end{equation}
This completes the proof.
\end{proof}

\subsection{The solutions of form $2x_0a+b\in Z(\bh_s)$}
In this subsection, we will find the necessary and sufficient conditions of  Equation IV having a solution
 such that  $2x_0+b\in Z(\bh_s)$.

  Suppose $x=x_0+x_1\bi+x_2\bj+x_3\bk\in SZ$ is a solution of  Eq.(\ref{linearx1}).
By Eq.(\ref{linearx1}), we have
\begin{equation}\label{inpax0}\left\langle 2x_0a+b,2x_0a+b \right\rangle=4x_0P_{ab}+I_b=0.\end{equation}
Based on this, we  divide our consideration into  two cases: $$P_{ab}\neq 0\mbox{ and }P_{ab}=0.$$
  \subsubsection{$P_{ab}\neq 0$}
If $P_{ab}\neq 0$ then by (\ref{inpax0}) we have \begin{equation}\label{slx0}x_0=\frac{-I_b}{4P_{ab}}.\end{equation}
 We reformulate  Eqs.(\ref{ae1}) and (\ref{ae2}) as
\begin{eqnarray}
 -x_1^2+x_2^2+x_3^2-b_1x_1+(b_2+2a_2x_0)x_2+(b_3+2a_3x_0)x_3+c_0+x_0^2&=&0,\label{1aeq1}\\
               2x_0x_1+(b_3+2a_3x_0)x_2-(b_2+2a_2x_0)x_3&=&-c_1-b_1x_0.\label{1aeq2}
 \end{eqnarray}
   Using  Eq.(\ref{ae1})$\times a_2$+Eq.(\ref{ae2})$\times a_3- $ Eq.(\ref{ae3})
    and Eq.(\ref{ae1})$\times a_3-$ Eq.(\ref{ae2})$\times a_2-$ Eq.(\ref{ae4}), we obtain
\begin{eqnarray}
  (-a_2b_1-b_3)x_1+(a_2b_2+a_3b_3)x_2+(a_2b_3-a_3b_2+b_1)x_3&=&c_2-c_0a_2-a_3c_1+(b_2-a_3b_1)x_0,\label{1aeq3}\\
 (-a_3b_1+b_2)x_1+(a_3b_2-a_2b_3-b_1)x_2+(a_2b_2+a_3b_3)x_3&=& c_3-c_0a_3+a_2c_1+(b_3+a_2b_1)x_0.\label{1aeq4}
\end{eqnarray}
Let $y=(x_1,x_2,x_3)^T$.  Eqs.(\ref{1aeq2})-(\ref{1aeq4}) can be expressed as
  \begin{equation}\label{lineareqs5} Ay=u,\end{equation}
 where
    \begin{equation}\label{matrixA}A=\left(
  \begin{array}{ccc}
  2x_0&b_3+2a_3x_0&-b_2-2a_2x_0\\
       -a_2b_1-b_3 & a_2b_2+a_3b_3 &a_2b_3-a_3b_2+b_1\\
    -a_3b_1+b_2 & a_3b_2-a_2b_3-b_1 & a_2b_2+a_3b_3 \\
    \end{array}
\right)\end{equation} and  \begin{equation}\label{matrixu}u=\left(
                 \begin{array}{c}
                   -c_1-b_1x_0\\
                  t_1+(b_2-a_3b_1)x_0 \\
                  t_2+(b_3+a_2b_1)x_0\\
                   \end{array}
               \right).\end{equation}

\begin{pro}\label{prop5.2} Let $x_0=\frac{-I_b}{4P_{ab}}$ and $a_2^2+a_3^2=1$. Let  $A$ be given by (\ref{matrixA}). Then
 $$\det(A)=0.$$
\end{pro}

\begin{proof}
Let \begin{equation}\label{matrixB}B=\left(
  \begin{array}{ccc}
  2x_0&b_3&-b_2\\
       -a_2b_1-b_3 & a_2b_2+2a_3b_3+a_2a_3b_1 &-a_3b_2+a_3^2b_1 \\
    -a_3b_1+b_2 & -a_2b_3-a_2^2b_1& 2a_2b_2+a_3b_3-a_2a_3b_1 \\
    \end{array}
\right).\end{equation}
It is obvious that $B$ is obtained by performing  elementary  column transformations form $A$.
 It can be verified that $\det(B)=0$. Therefore $\det(A)=\det(B)=0$.
\end{proof}
 Let \begin{equation}\label{matrixM}M=\left(\begin{array}{cc}
        a_2b_2+a_3b_3 &a_2b_3-a_3b_2+b_1 \\
    a_3b_2-a_2b_3-b_1& a_2b_2+a_3b_3\\
  \end{array}\right)=\left(\begin{array}{cc}
        -P_{ab} &\delta \\
   -\delta& -P_{ab}\\
  \end{array}\right).\end{equation} Since $P_{ab}\neq 0$, the subdeterminant
  $$m=:\det(M)=P_{ab}^2+\delta^2>0.$$
    By Proposition \ref{prop5.2}, this means that  $rank(A)=2$. We reformulate  Eqs.(\ref{1aeq3}) and (\ref{1aeq4}) as
   \begin{equation}Mz=v,\end{equation}
  where $$z=(x_2,x_3)^T,v=\left(
                 \begin{array}{c}
                   t_1+(b_2-a_3b_1)x_0+(a_2b_1+b_3)x_1 \\
                  t_2+(a_2b_1+b_3)x_0+(a_3b_1-b_2)x_1\\
                   \end{array}
                 \right).$$

  Let
$$k_1:=-P_{ab}(a_2b_1+b_3)-\delta(a_3b_1-b_2)=2b_2\delta-a_3I_b,$$
$$k_2:=-P_{ab}(b_2-a_3b_1)-\delta(a_2b_1+b_3)=-2b_3\delta-a_2I_b$$
and
\begin{equation*}\Delta_1=\frac{-P_{ab}t_1-\delta t_2}{m},\Delta_2=\frac{\delta t_1-P_{ab}t_2}{m}.\end{equation*}
Note that
$$ m=P_{ab}^2+\delta^2=b_1^2+b_2^2+b_3^2+2a_2b_1b_3-2a_3b_1b_2=2b_1\delta-I_b$$
and $$k_1^2+k_2^2=m^2.$$
   Because
  $$M^{-1}=\frac{1}{m}\left(\begin{array}{cc}
       -P_{ab} & -\delta\\
   \delta&-P_{ab}\\
  \end{array}\right)\mbox{  and } z=M^{-1}v,$$
    we have
\begin{eqnarray}x_2&=&\frac{-P_{ab}[t_1+(b_2-a_3b_1)x_0+(a_2b_1+b_3)x_1]
-\delta[t_2+(a_2b_1+b_3)x_0+(a_3b_1-b_2)x_1]}{m}\nonumber\\
&=&\frac{-P_{ab}(a_2b_1+b_3)-\delta(a_3b_1-b_2)}{m}x_1+\frac{-P_{ab}(b_2-a_3b_1)-\delta(a_2b_1+b_3)}{m}x_0
+\frac{-P_{ab}t_1-\delta t_2}{m}\nonumber\\
&=&\frac{k_1}{m}x_1+\frac{k_2}{m}x_0+\Delta_1\label{x2s5}\end{eqnarray}
and
\begin{eqnarray}x_3&=&\frac{\delta[t_1+(b_2-a_3b_1)x_0+(a_2b_1+b_3)x_1]-P_{ab}[t_2+(a_2b_1+b_3)x_0+(a_3b_1-b_2)x_1]}{m}\nonumber\\
&=&\frac{\delta(a_2b_1+b_3)-P_{ab}(a_3b_1-b_2)}{m}x_1
+\frac{-P_{ab}(a_2b_1+b_3)-\delta(a_3b_1-b_2)}{m}x_0+\frac{\delta t_1-P_{ab}t_2}{m}\nonumber\\
&=&-\frac{k_2}{m}x_1+\frac{k_1}{m}x_0+\Delta_2.\label{x3s5}\end{eqnarray}
Substituting the above two formulas  in Eq.(\ref{1aeq2}), we  have
\begin{eqnarray*}\Big(2x_0+\frac{b_3k_1+b_2k_2+2a_3k_1x_0+2a_2k_2x_0}{m}\Big)x_1+F=0,
\end{eqnarray*}
where \begin{eqnarray}\label{condthm5.1}F=\frac{2a_3k_2-2a_2k_1}{m}x_0^2
+\Big(\frac{b_3k_2-b_2k_1}{m}+2a_3\Delta_1-2a_2\Delta_2+b_1\Big)x_0+b_3\Delta_1-b_2\Delta_2+c_1.\end{eqnarray}
Note that $$2x_0+\frac{b_3k_1+b_2k_2+2a_3k_1x_0+2a_2k_2x_0}{m}=0.$$ By the solvability of $Ay=u$, we should have $F=0$.
We remark that the fact that the coefficient of $x_1$ is zero
 is guaranteed by $\det(A)=0$ and $F=0$ is just a restatement of $rank(A)=rank(A,u)=2$.

Substituting  $x_2$ and $x_3$ of  (\ref{x2s5}) and (\ref{x3s5}) in Eq.(\ref{1aeq1}), we obtain
$$Rx_1+L=0,$$
where \begin{eqnarray}R=\frac{2k_1\Delta_1-2k_2\Delta_2+b_2k_1-b_3k_2+2(a_2k_1-a_3k_2)x_0-mb_1}{m}\end{eqnarray} and
\begin{eqnarray}L&=&b_2\Delta_1+b_3\Delta_2+\Delta_1^2+\Delta_2^2+c_0+\frac{2(a_2k_2+a_3k_1+m)}{m}x_0^2\nonumber\\
&&+\frac{(2k_2\Delta_1+2k_1\Delta_2+b_2k_2+b_3k_1+2a_2\Delta_1m+2a_3\Delta_2m)}{m}x_0.\end{eqnarray}
 If $R=0$  we should have $L=0$ and in this case, $x_1$ is arbitrary.
 If $R\neq 0$ then $$x_1=\frac{-L}{R}.$$

Summarizing our reasoning process, we figure out the following  conditions.

 \begin{defi}\label{def5.1}
For the coefficients $a,b,c$ in Equation IV  such that $P_{ab}\neq 0,$
we set \begin{equation}\label{x0s5thm}x_0=\frac{-I_b}{4P_{ab}},\end{equation}
\begin{equation}\label{k1k2mthm}k_1=2b_2\delta-a_3I_b,k_2=-2b_3\delta-a_2I_b, m=2b_1\delta-I_b,\end{equation}
\begin{equation}\label{Delta}\Delta_1=\frac{-P_{ab}t_1-\delta t_2}{m},\Delta_2=\frac{\delta t_1-P_{ab}t_2}{m},\end{equation}
 \begin{eqnarray}\label{thmR}R=\frac{2k_1\Delta_1-2k_2\Delta_2+b_2k_1-b_3k_2+2(a_2k_1-a_3k_2)x_0-mb_1}{m},\end{eqnarray}
   \begin{eqnarray}L&=&b_2\Delta_1+b_3\Delta_2+\Delta_1^2+\Delta_2^2+c_0+\frac{2(a_2k_2+a_3k_1+m)}{m}x_0^2\nonumber\\
&&+\frac{(2k_2\Delta_1+2k_1\Delta_2+b_2k_2+b_3k_1+2a_2\Delta_1m+2a_3\Delta_2m)}{m}x_0,\label{thmL}\end{eqnarray}
and
\begin{eqnarray}\label{condthm5.2F}F=\frac{2a_3k_2-2a_2k_1}{m}x_0^2
+\Big(\frac{b_3k_2-b_2k_1}{m}+2a_3\Delta_1-2a_2\Delta_2+b_1\Big)x_0+b_3\Delta_1-b_2\Delta_2+c_1.\end{eqnarray}
We say  $(a,b,c)$   satisfies  {\bf Condition B} if the following two conditions hold:
\begin{itemize}
		\item [(1)]  $F=0;$
\item [(2)] If $R=0$ then $L=0$.
\end{itemize}
 \end{defi}

Summarizing the previous results, we obtain the following theorem.
\begin{thm}\label{thm5.1}
 Equation IV with $P_{ab}\neq 0$  has a  solution $x\in SZ$ if and only if Condition B holds.
 Let $x_0,k_1,k_2,m,\Delta_1,\Delta_2,R,L,F$ be given by Definition \ref{def5.1}.
If Condition B holds, then we have the following cases:
\begin{itemize}
	\item [(1)]
 If $R\neq 0$  then Equation IV has a solution:
 $$x=x_0-\frac{L}{R}\bi+x_2\bj+x_3\bk,$$
where $$x_2= -\frac{k_1L}{mR}+\frac{k_2}{m}x_0+\Delta_1$$
and
 $$x_3=\frac{k_2L}{mR}+\frac{k_1}{m}x_0+\Delta_2.$$

\item [(2)]  If $R=0$ then Equation IV has solutions:
$$x=x_0+x_1\bi+(\frac{k_1}{m}x_1+\frac{k_2}{m}x_0+\Delta_1)\bj+(-\frac{k_2}{m}x_1+\frac{k_1}{m}x_0+\Delta_2)\bk,
 \forall x_1\in \br.$$
\end{itemize}
\end{thm}

\begin{exam}\label{exam5.1}
	Consider the quadratic equation $(1+\bj)x^2 +(\bi+2\bj+\bk)x-\frac{1}{4}+\frac{5}{2}\bi+\frac{3}{4}\bi+\frac{5}{2}\bk=
0$.
 That is, $a=1+\bj,b=\bi+2\bj+\bk$ and $c=-\frac{1}{4}+\frac{5}{2}\bi+\frac{3}{4}\bi+\frac{5}{2}\bk$. In this case
$$x_0=-\frac{1}{2},k_1=8,k_2=0,\Delta_1=-1,\Delta_2=\frac{3}{2},m=8,R=-2,L=2, F=0.$$
Therefore $(a,b,c)$ satisfies  {\it Condition B} and $x_1=-\frac{L}{R}=1, x_2=0,x_3=1$.
 Thus $$x=-\frac{1}{2}+\bi+\bk$$ is a solution of the given quadratic equation.
\end{exam}

\begin{exam}\label{exam5.2}
	Consider the quadratic equation $(1+\bj)x^2 +(\bi+\bj)x-1+\bi= 0$.
That is, $a=1+\bj,b=\bi+\bj$ and $c=-1+\bi$. In this case $$x_0=0,k_1=2,k_2=0,\Delta_1=0,\Delta_2=1,m=2,R=L=0,F=0.$$
Therefore $(a,b,c)$ satisfies  {\it Condition B}. In this case $x_1$ is arbitrary, $x_2=x_1,x_3=1$.
Thus $$x=x_1\bi+x_1\bj+\bk,\forall x_1\in \br$$  are solutions of   the given quadratic equation.
\end{exam}

\subsubsection{$P_{ab}=0$}
In this subsection, we will find the necessary and sufficient conditions of  Equation IV with $P_{ab}=0$
 having a solution $x\in SZ$.

Suppose that Equation IV with $P_{ab}=0$  has a solution $x\in SZ$.
By Proposition \ref{prop5.1}, under the condition  $P_{ab}=0$ and $I_a=0$, we have
      \begin{eqnarray}
 (b_2-a_3b_1)x_0+(a_2b_1+b_3)x_1+(a_3b_2-a_2b_3-b_1)x_3+t_1=0,\label{2new3}\\
 (a_2b_1+b_3)x_0+(a_3b_1-b_2)x_1+(a_2b_3-a_3b_2+b_1)x_2+t_2=0.\label{2new4}
\end{eqnarray}
 Since
$$\left\langle 2x_0a+b,2x_0a+b \right\rangle=4x_0P_{ab}+I_b=0.$$  By our assumption $P_{ab}=0$, we must have $I_b=0$.
 By $P_{ab}=I_b=0$, we have $$b_1^2-(a_3b_2-a_2b_3)^2=0.$$
 Thus we have $\delta=2b_1$ or $\delta=0$.
  We divide our consider into two subcases: $$\delta=2b_1\mbox{ and }\delta=0.$$

   We begin with the case $\delta=2b_1$, that is $b_1=-a_3b_2+a_2b_3$.

If $b_1=-a_3b_2+a_2b_3$, then by $P_{ab}=0$ and $I_a=0$ we have $b_2-a_3b_1=2b_2,a_2b_1+b_3=2b_3.$
 Thus \begin{equation}\label{b3b2c1}b_3=a_2b_1,b_2=-a_3b_1.\end{equation} By our assumption $b\neq 0$, we have $b_1\neq 0$.
  Hence Eqs. (\ref{2new3}) and (\ref{2new4}) become
     \begin{eqnarray*}
  2b_2x_0+2b_3x_1-2b_1x_3+t_1=0,\\
 2b_3x_0-2b_2x_1+2b_1x_2+t_2=0.
\end{eqnarray*}
  From the above and Eq.(\ref{b3b2c1}), we get
      \begin{eqnarray}
   x_2&=&-a_2x_0-a_3x_1-\frac{t_2}{2b_1},\label{s5newx2}\\
  x_3&=&-a_3x_0+a_2x_1+\frac{t_1}{2b_1}.\label{s5newx3}
\end{eqnarray}
Substituting   the above two formulas of $x_2$ and $x_3$ in Eq.(\ref{ae2}),
  we obtain
  \begin{eqnarray*}
-\frac{a_2t_1+a_3t_2}{b_1}x_0-\frac{a_2t_2-a_3t_1}{2}+c_1=0.
   \end{eqnarray*}
  If $a_2t_1+a_3t_2=0$ then we must have  $-\frac{a_2t_2-a_3t_1}{2}+c_1=0$ and in this case  $x_0$ is arbitrary.
If $a_2t_1+a_3t_2\neq 0$ then $$x_0=\frac{(a_3t_1-a_2t_2+2c_1)b_1}{2(a_2t_1+a_3t_2)}.$$

  Substituting $x_2$ and $x_3$ of  (\ref{s5newx2}) and (\ref{s5newx3}) in Eq.(\ref{ae1}),
    we obtain
\begin{eqnarray*}
\frac{a_2t_1+a_3t_2}{b_1}x_1+\frac{t_1^2+t_2^2}{4b_1^2}+\frac{a_2t_1+a_3t_2}{2}+c_0=0.
   \end{eqnarray*}
If $a_2t_1+a_3t_2=0$ then we need $\frac{t_1^2+t_2^2}{4b_1^2}+c_0=0$ and $x_1$ is arbitrary.
If $a_2t_1+a_3t_2\neq 0$ then $$x_1=-\frac{t_1^2+t_2^2+2b_1^2(a_2t_1+a_3t_2)+4b_1^2c_0}{4b_1(a_2t_1+a_3t_2)}.$$

By the above reasoning process, we figure out the following condition.

 \begin{defi}\label{def5.2}
  For the coefficients $a,b,c$ in Equation IV  such that $P_{ab}=0$ and $I_b=0$,
     $(a,b,c)$   satisfies  {\bf Condition C} if the following two conditions hold:
\begin{itemize}
		\item [(1)]  $\delta=2b_1;$
\item [(2)]  If $a_2t_1+a_3t_2=0$ then  $a_2t_2-a_3t_1-2c_1=0$ and  $t_1^2+t_2^2+4b_1^2c_0=0$.
\end{itemize}
 \end{defi}

Summarizing the previous results, we obtain the following theorem.
\begin{thm}\label{thm5.2}
 Equation IV with $P_{ab}=0$ and $\delta=2b_1$ has a  solution $x\in SZ$  if and only if Condition C holds.
If Condition C holds, then we have the following cases.
	\begin{itemize} \item [(1)] If  $a_2t_1+a_3t_2\neq 0$ then   Equation IV  has a solution
 $$x=x_0+x_1\bi+x_2\bj+x_3\bk,$$
where
\begin{eqnarray*}\label{s5x0}x_0=\frac{(a_3t_1-a_2t_2+2c_1)b_1}{2(a_2t_1+a_3t_2)}
   \end{eqnarray*}
 \begin{eqnarray*}\label{s5x1}x_1=-\frac{t_1^2+t_2^2+2b_1^2(a_2t_1+a_3t_2)+4b_1^2c_0}{4b_1(a_2t_1+a_3t_2)},
   \end{eqnarray*}
and
  \begin{eqnarray*}
  x_2&=&-a_2x_0-a_3x_1-\frac{t_2}{2b_1},\\
  x_3&=&-a_3x_0+a_2x_1+\frac{t_1}{2b_1}.
\end{eqnarray*}
 \item [(2)]If  $a_2t_1+a_3t_2=0$
 then   Equation IV  has solutions
 $$x=x_0+x_1\bi+x_2\bj+x_3\bk,\forall x_0,x_1\in \br,$$
where
  \begin{eqnarray*}
   x_2&=&-a_2x_0-a_3x_1-\frac{t_2}{2b_1},\\
  x_3&=&-a_3x_0+a_2x_1+\frac{t_1}{2b_1}.
\end{eqnarray*}
\end{itemize}

\end{thm}

\begin{exam}\label{examthm5.3}
	Consider the quadratic equation $(1+\bj)x^2 +(\bi+\bk)x+1-\bi=0$.
That is, $a=1+\bj,b=\bi+\bk$ and $c=1-\bi$. In this case $t_1=t_2=-1$ and $a_2t_1+a_3t_2=-1$.
The equation $ax^2+bx+c=0$ has a solution
$x=\frac{1}{2}+\bi+\frac{1}{2}\bk.$
\end{exam}

\begin{exam}\label{examthm5.4}
	Consider the quadratic equation $(1+\bj)x^2 +(\bi+\bk)x-1+\bi-\bj+\bk=0$.
That is, $a=1+\bj,b=\bi+\bk$ and $c=-1+\bi-\bj+\bk$.
In this case $t_1=0,t_2=2$ and $a_2t_1+a_3t_2=0$.  The equation $ax^2+bx+c=0$ has  solutions
$x=x_0+x_1\bi-(1+x_0)\bj+x_1\bk,\forall x_0,x_1\in \br.$
\end{exam}

We now consider the  second case $\delta=0$, that is, $b_1=a_3b_2-a_2b_3$.  If $b_1=a_3b_2-a_2b_3$ then by $P_{ab}=0$ and
$I_a=0$
we have $$b_2-a_3b_1=a_2(a_2b_2+a_3b_3)=0,\,a_2b_1+b_3=a_3(a_2b_2+a_3b_3)=0.$$
So we have
\begin{equation}\label{b3b2c2}b_2=a_3b_1,b_3=-a_2b_1.\end{equation}
From the above formulas, Eqs.(\ref{2new3}) and (\ref{2new4}) imply that
\begin{equation}\label{ac1}c_2-a_2c_0-a_3c_1=0,\,c_3+a_2c_1-a_3c_0=0.\end{equation}
By $I_a=0$ and the above two conditions, we have
\begin{equation}\label{ac2}P_{ac}=\left\langle a,c\right\rangle=c_0-a_2c_2-a_3c_3=0.\end{equation}
From this, we get
\begin{equation}\label{ac3}c_1=a_3c_2-a_2c_3.\end{equation}
Eqs. (\ref{ac1})-(\ref{ac3}) is equivalent to the condition \begin{equation}\label{acend}ac=2c.\end{equation}
Under the condition $P_{ab}=I_a=I_b=0,b_2=a_3b_1,b_3=-a_2b_1$ and $ac=2c$, we have
\begin{center}Eq.(\ref{ae3})$=$ Eq.(\ref{ae1})$\times a_2+$Eq.(\ref{ae2})$\times a_3$ and
Eq.(\ref{ae4})$=$ Eq.(\ref{ae1})$\times a_3-$Eq.(\ref{ae2})$\times a_2$.\end{center}
Hence in this case Equation IV  only has two independent equalities Eqs.(\ref{ae1}) and (\ref{ae2}), which can be reformulated
as
\begin{eqnarray}
x_0^2+2(a_2x_2+a_3x_3)x_0-x_1^2+x_2^2+x_3^2-b_1x_1+a_3b_1x_2-a_2b_1x_3+c_0=0,\label{ab0eq1}\\
x_0(2x_1-2a_2x_3+2a_3x_2+b_1)=b_1(a_2x_2+a_3x_3)-c_1.\label{ab0eq2}
\end{eqnarray}
  These are underdetermined system of equations.

Note that $b_1\neq 0$.
By  Eq.(\ref{ab0eq2}), if $x_0=0$ then  $$a_2x_2+a_3x_3=\frac{c_1}{b_1}.$$
we treated the cases $a_2=0$ and $a_2\neq 0$, respectively.

If $a_2=0$ then $a_3\neq 0$ and therefore $$x_3=\frac{c_1}{a_3b_1}.$$ Note that $a_3^2=1$.
Substituting $x_0=0$ and $x_3=\frac{c_1}{a_3b_1}$ in Eq.(\ref{ab0eq1}), we obtain
$$x_2^2+a_3b_1x_2+\frac{c_1^2}{b_1^2}+c_0-x_1^2-b_1x_1=0.$$
So we have a solution $$x=x_1\bi+x_2\bj+\frac{c_1}{a_3b_1}\bk,$$
where $$x_2=\frac{-a_3b_1\pm \sqrt{b_1^2-4(\frac{c_1^2}{b_1^2}+c_0-x_1^2-b_1x_1)}}{2}$$
and $x_1\in \br$ satisfies
$$\frac{c_1^2}{b_1^2}+c_0-\frac{b_1^2}{4}-x_1^2-b_1x_1\leq 0.$$

If $a_2\neq 0$ then \begin{equation}\label{addx2}x_2=\frac{c_1}{a_2b_1}-\frac{a_3}{a_2}x_3.\end{equation}
Substituting $x_0=0$ and the above formula 
in Eq.(\ref{ab0eq1}), we obtain
$$x_1^2+b_1x_1+t=0,$$
where $$t=-\frac{1}{a_2^2}x_3^2+(\frac{2a_3c_1+a_2b_1^2}{a_2^2b_1})x_3-(c_0+\frac{a_3c_1}{a_2}+\frac{c_1^2}{a_2^2b_1^2}).$$
Hence $x_1$ can be expressed by $x_3$ as
\begin{equation}\label{addx1}x_1=\frac{-b_1\pm \sqrt{b_1^2-4t}}{2}\end{equation}
and $x_3\in \br$ satisfies
\begin{equation}\label{addb1}b_1^2-4t=\frac{4}{a_2^2}[x_3^2-(\frac{2a_3c_1+a_2b_1^2}{b_1})x_3+\frac{4(a_2^2b_1^2c_0+a_2b_1^2a_3c_1+c_1^2)+b_1^4a_2^2}{4b_1^2}]
\geq 0.\end{equation}
So we have  solutions $$x=x_1\bi+x_2\bj+x_3\bk,$$
where $x_1,x_2$ are given by (\ref{addx1}) and (\ref{addx2}), and $x_3\in \br$ satisfies (\ref{addb1}).

If $x_0\neq 0$ then from  Eq.(\ref{ab0eq2}) we have
 $$2x_1-2a_2x_3+2a_3x_2+b_1=\frac{a_3b_1x_3+a_2b_1x_2-c_1}{x_0}.$$
From this, we get $$x_1=(\frac{a_2b_1}{2x_0}-a_3)x_2+(\frac{a_3b_1}{2x_0}+a_2)x_3-\frac{c_1+b_1x_0}{2x_0}.$$
Substituting the above formula in  Eq.(\ref{ab0eq1}) and rearranging the equation, we obtain
\begin{eqnarray*}
&&x_0^4+2(a_2x_2+a_3x_3)x_0^3+[(a_2x_2+a_3x_3)^2+b_1(a_3x_2-a_2x_3)+c_0+\frac{b_1^2}{4}]x_0^2\\
&&+[a_2a_3b_1(x_2^2-x_3^2)+b_1(a_3^2-a_2^2)x_2x_3+c_1(a_2x_3-a_3x_2)]x_0-\frac{[b_1(a_2x_2+a_3x_3)-c_1]^2}{4}=0.
\end{eqnarray*}
Let \begin{eqnarray*}
f(z)&=&z^4+2(a_2x_2+a_3x_3)z^3+[(a_2x_2+a_3x_3)^2+b_1(a_3x_2-a_2x_3)+c_0+\frac{b_1^2}{4}]z^2\\
&&+[a_2a_3b_1(x_2^2-x_3^2)+b_1(a_3^2-a_2^2)x_2x_3+c_1(a_2x_3-a_3x_2)]z-\frac{[b_1(a_2x_2+a_3x_3)-c_1]^2}{4}=0.
\end{eqnarray*}
Then $$f(0)=-\frac{[b_1(a_2x_2+a_3x_3)-c_1]^2}{4}\leq 0,\lim_{z\to +\infty} f(z)=+\infty,\lim_{z\to -\infty} f(z)=+\infty.$$

If $a_2x_2+a_3x_3\neq \frac{c_1}{b_1}$ then  $f(0)<0$,  and $f(z)=0$  has at least two real  solutions  $z_1\in (-\infty,0)$
and $z_2\in (0,\infty)$.
 Let $T\in \br$ be a solution of $f(z)=0$ with $a_2x_2+a_3x_3\neq \frac{c_1}{b_1}$.
 Then Equation IV has a solution
 $$x=T+x_1\bi+x_2\bj+x_3\bk,$$
  where
$$x_1=(\frac{a_2b_1}{2T}-a_3)x_2+(\frac{a_3b_1}{2T}+a_2)x_3-\frac{c_1+b_1T}{2T}.$$

If \begin{equation}\label{ab0eqx23}a_2x_2+a_3x_3=\frac{c_1}{b_1},\end{equation} then  by (\ref{ab0eq2}) and our assumption
$x_0\neq 0$
 we have
\begin{eqnarray}
               2x_1-2a_2x_3+2a_3x_2+b_1=0.\label{ab0eq3}
  \end{eqnarray}
By (\ref{ab0eqx23}) and (\ref{ab0eq3}), we obtain that
\begin{eqnarray*}
 x_2=\frac{a_2c_1}{b_1}-\frac{a_3b_1}{2}-a_3x_1,\label{ab0eq4}\\
 x_3=\frac{a_3c_1}{b_1}+\frac{a_2b_1}{2}+a_2x_1.\label{ab0eq5}
  \end{eqnarray*}
Substituting the above formulas in (\ref{ab0eq1}), we obtain
\begin{eqnarray}
             x_0^2+\frac{2c_1}{b_1}x_0-b_1x_1+\frac{c_1^2}{b_1^2}-\frac{b_1^2}{4}+c_0=0.
  \end{eqnarray}
Hence
\begin{eqnarray*}
            x_1= \frac{1}{b_1}x_0^2+\frac{2c_1}{b_1^2}x_0+\frac{c_1^2}{b_1^3}-\frac{b_1}{4}+\frac{c_0}{b_1}.\label{s5newx1}
  \end{eqnarray*}
Form the above description,  Equation IV has  solutions   $$x=x_0+x_1\bi+x_2\bj+x_3\bk,\forall x_0\neq 0,$$
where $x_1,x_2,x_3$ are expressed  by formulas containing $x_0$ as above.

Summarizing the previous results, we obtain the following theorem.

\begin{thm}\label{thm5.3}
 Equation IV with $P_{ab}=0$ and $\delta=0$ has a  solution $x\in SZ$  if and only if
 $ac=2c$.
If  Equation IV is solvable then we have the following cases:
\begin{itemize}
 \item [(1)] Case $x_0=0$:
\begin{itemize}
  \item [(1.1)] if  $a_2=0$ then   Equation IV has  solutions:$$x=x_1\bi+x_2\bj+\frac{c_1}{a_3b_1}\bk,$$
      where $$x_2=\frac{-a_3b_1\pm \sqrt{b_1^2-4(\frac{c_1^2}{b_1^2}+c_0-x_1^2-b_1x_1)}}{2}$$
      and $x_1$ is real numbers satisfies
      $$x_1^2+b_1x_1+\frac{b_1^2}{4}-\frac{c_1^2}{b_1^2}-c_0\geq 0.$$

  \item [(1.2)] if   $a_2\neq 0$ then   Equation IV has  solutions:$$x=x_1\bi+x_2\bj+x_3\bk,$$
      where $x_3\in \br$  satisfies $$w=x_3^2-(\frac{2a_3c_1+a_2b_1^2}{b_1})x_3
      +\frac{4(a_2^2b_1^2c_0+a_2b_1^2a_3c_1+c_1^2)+b_1^4a_2^2}{4b_1^2}\ge 0$$
      and
      $$x_1=\frac{-b_1}{2}\pm \frac{\sqrt{w}}{a_2},$$
      $$x_2=\frac{c_1}{a_2b_1}-\frac{a_3}{a_2}x_3.$$
     \end{itemize}
\item [(2)] Case $x_0\neq 0$:
\begin{itemize}
  \item [(2.1)] $a_2x_2+a_3x_3\neq \frac{c_1}{b_1}$:
  Equation IV has  solutions:
             $$x=T+x_1\bi+x_2\bj+x_3\bk,$$  where
            $T$ be a real solution  of the following equation:  \begin{eqnarray*}
 z^4+2(a_2x_2+a_3x_3)z^3+[(a_2x_2+a_3x_3)^2+b_1(a_3x_2-a_2x_3)+c_0+\frac{b_1^2}{4}]z^2\\
 +[a_2a_3b_1(x_2^2-x_3^2)+b_1(a_3^2-a_2^2)x_2x_3+c_1(a_2x_3-a_3x_2)]z-\frac{[b_1(a_2x_2+a_3x_3)-c_1]^2}{4}=0
   \end{eqnarray*}
   and
$$x_1=(\frac{a_2b_1}{2T}-a_3)x_2+(\frac{a_3b_1}{2T}+a_2)x_3-\frac{c_1+b_1T}{2T}.$$

\item [(2.2)] $a_2x_2+a_3x_3=\frac{c_1}{b_1}$:
  Equation IV has  solutions:
             $$x=x_0+x_1\bi+x_2\bj+x_3\bk,\forall x_0\neq 0,$$  where
     \begin{eqnarray*}
            x_1&=& \frac{1}{b_1}x_0^2+\frac{2c_1}{b_1^2}x_0+\frac{c_1^2}{b_1^3}-\frac{b_1}{4}+\frac{c_0}{b_1},\\
  x_2&=&\frac{a_2c_1}{b_1}-\frac{a_3b_1}{2}-a_3x_1,\\
 x_3&=&\frac{a_3c_1}{b_1}+\frac{a_2b_1}{2}+a_2x_1.
  \end{eqnarray*}
  \end{itemize}
  \end{itemize}
\end{thm}

\begin{exam}\label{exam5.5}
	Consider the quadratic equation \begin{equation}\label{exam5.5eq}(1+\bk)x^2 +(\bi+\bj)x+1+2\bi+2\bj+\bk= 0.\end{equation}
That is, $a=1+\bk,b=\bi+\bj$ and $c=1+2\bi+2\bj+\bk$.  Then we have the following cases:
\begin{itemize}
  \item[(1.1)]
  Eq.(\ref{exam5.5eq}) has  the following solutions
  $$x=x_1\bi-\Big(\frac{1}{2}\pm \sqrt{x_1^2+x_1-\frac{19}{4}}\Big)\bj+2\bk,$$
  where $x_1$ is arbitrary but satisfies $x_1^2+x_1-\frac{19}{4}\geq 0$.
  \item[(2.1)]  Eq.(\ref{exam5.5eq}) has  the following solutions

            $$x=T+x_1\bi+x_2\bj+x_3\bk,\forall x_3\neq 2, x_2 \in \br$$  where
            $T$ be a real solution  of the following equation:  \begin{eqnarray}\label{examp5.5x0}
 z^4+2x_3z^3+(x_3^2+x_2+\frac{5}{4})z^2+(x_2x_3-2x_2)z-\frac{(x_3-2)^2}{4}=0
   \end{eqnarray}
   and
$$x_1=-x_2+\frac{1}{2T}x_3-\frac{2+T}{2T}.$$
For example, if we take $x_2=x_3=1$, then Eq. (\ref{examp5.5x0}) has real solution $T_1=0.3914$ and $T_2=-0.1675$.
So we have solutions
$$x=0.3914-2.7773\bi+\bj+\bk,\mbox{  and  } x=-0.1675+1.4857\bi+\bj+\bk.$$
\item[(2.2)]  When $x_3=2$,  Eq.(\ref{exam5.5eq}) has  the following solutions
             $$x=x_0+x_1\bi-(x_1+\frac{1}{2})\bj+2\bk,\forall x_0\neq 0,$$  where
     \begin{eqnarray*}
            x_1= x_0^2+4x_0+\frac{19}{4}.
  \end{eqnarray*}
  \end{itemize}
\end{exam}

\begin{exam}\label{exam5.6}
	Consider the quadratic equation \begin{equation}\label{exam5.6eq}(1+\bj)x^2 +(-\bi+\bj)x-1+\bi-\bj-\bk= 0.\end{equation}
That is, $a=1+\bj,b=-\bi+\bk$ and $c=-1+\bi-\bj-\bk$.  Then we have the following cases.
\begin{itemize}
  \item[(1.2)] Eq.(\ref{exam5.6eq}) has  the following solutions
  $$x=(1+x_3)\bi-\bj+x_3\bk\mbox{ and  } x=-x_3\bi-\bj+x_3\bk,  \forall x_3\in \br.$$
  \item[(2.1)] Eq.(\ref{exam5.6eq}) has  the following solutions
            $$x=T+x_1\bi+x_2\bj+x_3\bk,\forall x_2\neq -1,x_3\in \br$$  where
            $T$ be a real solution  of the following equation:  \begin{eqnarray}\label{examp5.6x0}
 z^4+2x_2z^3+(x_2^2+x_3-\frac{3}{4})z^2+(x_2x_3+x_3)z-\frac{(x_2+1)^2}{4}=0
   \end{eqnarray}
   and
$$x_1=-\frac{1}{2T}x_2+x_3+\frac{T-1}{2T}.$$
For example, if we take $x_2=x_3=1$, then Eq. (\ref{examp5.6x0}) has real solution $x_0=-2$ and $x_0=0.362$.
So we have solutions
$$x=-2+2\bi+\bj+\bk,\mbox{  and  } x=0.3620-1.2621\bi+\bj+\bk.$$
\item[(2.2)]  When $x_2=-1$, Eq.(\ref{exam5.6eq}) has  the following solutions
             $$x=x_0+x_1\bi-\bj+(x_1-\frac{1}{2})\bk,\forall x_0\neq 0,$$  where
     \begin{eqnarray*}
            x_1= -x_0^2+2x_0+\frac{1}{4}.
  \end{eqnarray*}
 \end{itemize}
\end{exam}

\subsection{The solutions of form $2x_0a+b\in \bh_s-Z(\bh_s)$}

In this subsection we consider  Equation IV  for $SI$.

 Since $a=1+a_2\bj+a_
3\bk\in Z(\bh_s)$ and $I_a=0$, the real nonlinear system (\ref{rsym1}) is simplified to
\begin{eqnarray}
 N(2TP_{ab}+I_b+2P_{ac})-I_c=0,\label{1ieq1}\\
2P_{ab}T^2+(2P_{ac}+I_b)T-2NP_{ab}+2P_{bc}=0.\label{1ieq2}
\end{eqnarray}

We treat the case $P_{ab}\neq 0$ and $P_{ab}\neq 0$ separately.

\subsubsection{$P_{ab}\neq 0$}

\begin{thm}\label{thm5.4}Equation IV with $P_{ab}\neq 0$ has a solution
$$x=(Ta+b)^{-1}(aN-c),$$
where $T$ is a real solution of the following  cubic equation
\begin{equation}\label{Tab}4P_{ab}^2T^3+[4P_{ab}(2P_{ac}+I_b)]T^2+[4P_{ab}P_{bc}
+(2P_{ac}+I_b)^2]T+2P_{bc}(2P_{ac}+I_b)-2P_{ab}I_c=0. \end{equation}
and
\begin{equation}\label{Nab}N=\frac{2P_{ab}T^2+(2P_{ac}+I_b)T+2P_{bc}}{2P_{ab}}. \end{equation}
\end{thm}
\begin{proof}

If $P_{ab}\neq 0$ then by (\ref{1ieq2}) we get
\begin{equation}\label{Nab2}N=\frac{2P_{ab}T^2+(2P_{ac}+I_b)T+2P_{bc}}{2P_{ab}}. \end{equation}
Substituting the above $N$  in   (\ref{1ieq1}), we obtain
\begin{equation}\label{Tab2}4P_{ab}^2T^3+[4P_{ab}(2P_{ac}+I_b)]T^2+[4P_{ab}P_{bc}+(2P_{ac}+I_b)^2]T
+2P_{bc}(2P_{ac}+I_b)-2P_{ab}I_c=0. \end{equation}
Let $T$ be a real solution of the above cubic equation. Then the corresponding solution is
$$x=(Ta+b)^{-1}(aN-c).$$
\end{proof}

\begin{exam}\label{exam5.7}
	Consider the quadratic equation $(1+\bj)x^2 +(\bi+\bj)x-1+\bi=0$.
That is, $a=1+\bj,b=\bi+\bj$ and $c=-1+\bi$.   $P_{ab}=-1$. In this case
$T=-2, N=1$ and $$x=(Ta+b)^{-1}(aN-c)=-1.$$
\end{exam}
Combining this example with Example 5.2, we know that the set of solution of the equation $$(1+\bj)x^2 +(\bi+\bj)x-1+\bi=0$$
is
$$\{-1\}\cup \{x=x_1\bi+x_1\bj+\bk,\forall x_1\in \br\}.$$

\subsubsection{$P_{ab}=0$}
\begin{thm}\label{thm5.5}Equation IV with $P_{ab}=0$ and $I_b+2P_{ac}\neq 0$
is solvable  and   $$x=(Ta+b)^{-1}(aN-c),$$  where $$N=\frac{I_c}{I_b+2P_{ac}},T= \frac{-2P_{bc}}{I_b+2P_{ac}}.$$
\end{thm}
\begin{proof}
Since $\left\langle 2x_0a+b,2x_0a+b \right\rangle=4x_0P_{ab}+I_b\neq 0$ and $P_{ab}=0$,
  we have $I_b\neq 0$.
       If $P_{ab}=0$ then by (\ref{1ieq1}) and (\ref{1ieq2}),  $(T,N)$ satisfies the real system
      \begin{eqnarray}
     N(I_b+2P_{ac})=I_c,\label{ieq1}\\
     (2P_{ac}+I_b)T=-2P_{bc}.\label{ieq2}
      \end{eqnarray}
     If $I_b+2P_{ac}\neq 0$ then $$N=\frac{I_c}{I_b+2P_{ac}},T= \frac{-2P_{bc}}{I_b+2P_{ac}}.$$
So the corresponding solution is
      $$x=(Ta+b)^{-1}(aN-c).$$
\end{proof}
\begin{exam}\label{exam5.8}
	Consider the quadratic equation $(1+\bj)x^2 +(2\bi+\bk)x+1+\bi+2\bj+\bk= 0$.
 That is, $a=1+\bj,b=2\bi+\bk$ and $c=1+\bi+2\bj+\bk$.  $P_{ab}=0, I_b+2P_{ac}=1$. In this case
$T=-2, N=-3$ and $$x=(Ta+b)^{-1}(aN-c)=-1+\frac{17}{3}\bi+\frac{1}{3}\bj+6\bk.$$
  \end{exam}

To treat  the case of $I_b+2P_{ac}=0$, we need the following  proposition.

\begin{pro}\label{prop5.3}
For the coefficients $a,b,c$ in Equation IV, we assume that
$$P_{ab}=0,I_a=0,I_c=0,P_{bc}=0,I_b+2P_{ac}=0, I_b\neq 0.$$
Then we have
\begin{equation}\label{formu1}\frac{(a_2b_1+b_3)^2+(b_2-a_3b_1)^2}{\delta^2}=1,\end{equation}
\begin{equation}\label{formu2}\frac{a_3(b_2-a_3b_1)-a_2(a_2b_1+b_3)}{\delta}=-1,\end{equation}
\begin{equation}\label{formu3}\frac{a_3(a_2b_1+b_3)+a_2(b_2-a_3b_1)}{\delta}=0,\end{equation}
\begin{equation}\label{formu4}\frac{2t_2(a_2b_1+b_3)+2t_1(b_2-a_3b_1)}{\delta^2}
+\frac{b_3(b_2-a_3b_1)-b_2(a_2b_1+b_3)+2a_3t_1-2a_2t_2}{\delta}=0,\end{equation}
\begin{equation}\label{formu5}\frac{2t_2(a_3b_1-b_2)+2t_1(a_2b_1+b_3)}{\delta^2}
+\frac{(a_2b_1+b_3)b_3-b_2(a_3b_1-b_2)}{\delta}-b_1=0.\end{equation}
\end{pro}

\begin{proof}
By $a_2^2+a_3^2=1$ and $a_2b_2+a_3b_3=0$, we can easily verify Eqs.(\ref{formu1})-(\ref{formu3}).
Noting that $b_3(b_2-a_3b_1)-b_2(a_2b_1+b_3)=-b_1(a_2b_2+a_3b_3)=0$, $a_2b_1+b_3-a_2\delta=a_3(a_2b_2+a_3b_3)=0$
 and   $b_2-a_3b_1+a_3\delta=a_2(a_2b_2+a_3b_3)=0$,
  we have
  $$2t_2(a_2b_1+b_3)+2t_1(b_2-a_3b_1)+(2a_3t_1-2a_2t_2)\delta=2(a_2b_1+b_3-a_2\delta)t_2+2(b_2-a_3b_1+a_3\delta)t_1=0.$$
This proves Eq.(\ref{formu4}).
It is obvious that
$$\frac{(a_2b_1+b_3)b_3-b_2(a_3b_1-b_2)}{\delta}-b_1=\frac{b_2^2+b_3^2-b_1^2}{\delta}=\frac{-I_b}{\delta}.$$
By $a_2^2+a_3^2=1$ and $a_2b_2+a_3b_3=0$, we have
$$b_3t_1-b_2t_2+P_{ac}(a_2b_3-a_3b_2)=0.$$
Noting $a_2t_1+a_3t_2=-P_{ac}$ and $-I_b=2P_{ac}$,
 we have $$t_2(a_3b_1-b_2)+t_1(a_2b_1+b_3)+P_{ac}\delta=(a_3t_2+a_2t_1)b_1+b_3t_1-b_2t_2+P_{ac}(a_2b_3-a_3b_2+b_1)=0.$$
This proves Eq.(\ref{formu5}).
\end{proof}

\begin{thm}\label{thm5.6}Consider Equation IV with $P_{ab}=0$ and $I_b+2P_{ac}=0$.
Let \begin{equation}\label{condpab02} F=t_1^2+t_2^2+(b_3t_1-b_2t_2)\delta+c_0\delta^2.\end{equation}
 Equation IV
is solvable if only if $F=0$.
 If $F=0$ then  Equation IV has solutions $$x=x_0+x_1\bi+x_2\bj+x_3\bk,\forall x_0,x_1\in \br,$$  where
   \begin{eqnarray}
  \label{sx2e} x_2&=&-\frac{t_2}{\delta}-\frac{(a_2b_1+b_3)}{\delta}x_0-\frac{(a_3b_1-b_2)}{\delta}x_1,\\
  \label{sx3e}x_3&=&\frac{t_1}{\delta}+\frac{(b_2-a_3b_1)}{\delta}x_0+\frac{(a_2b_1+b_3)}{\delta}x_1.
\end{eqnarray}
\end{thm}

\begin{proof}
Suppose that there is a solution $x\in SI$.
By Eq.(\ref{ieq1}) and Eq.(\ref{ieq2}), if $I_b+2P_{ac}=0$ then $$I_c=0,P_{bc}=0,I_b\neq 0.$$
In this special case, although $2x_0a+b\in \bh_s-Z(\bh_s)$,
 however Eq.(\ref{ieq1}) and Eq.(\ref{ieq2})  provide no information about $N$ and $T$. So we return to the original equation.

  By Proposition \ref{prop5.1}, under the condition  $P_{ab}=0$ and $I_a=0$, we have
      \begin{eqnarray}
  (b_2-a_3b_1)x_0+(a_2b_1+b_3)x_1-\delta x_3+t_1=0,\label{6new3}\\
 (a_2b_1+b_3)x_0+(a_3b_1-b_2)x_1+\delta x_2+t_2=0.\label{6new4}
\end{eqnarray}
 Since $P_{ab}=0,a_2^2+a_3^2=1$ and $I_b\neq 0$,
  we obtain $$b_1^2-(a_3b_2-a_2b_3)^2=b_1^2-b_2^2-b_3^2+(a_2b_2+a_3b_3)^2=I_b\neq 0.$$
  This means $\delta=a_2b_3-a_3b_2+b_1\neq 0$.
  So we have
      \begin{eqnarray*}
   x_2=-\frac{(a_2b_1+b_3)}{\delta}x_0-\frac{(a_3b_1-b_2)}{\delta}x_1-\frac{t_2}{\delta},\\
  x_3=\frac{(b_2-a_3b_1)}{\delta}x_0+\frac{(a_2b_1+b_3)}{\delta}x_1+\frac{t_1}{\delta}.
\end{eqnarray*}
  Substituting  the above two formulas of $x_2$ and $x_3$ in Eq. (8),
   that is,  $$x_0^2-x_1^2+x_2^2+x_3^2+2a_2x_0x_2+2a_3x_0x_3-b_1x_1+b_2x_2+b_3x_3+c_0=0,$$
  we obtain
   \begin{eqnarray*}
  &&\Big[1+\frac{(a_2b_1+b_3)^2+(b_2-a_3b_1)^2}{\delta^2}+2\frac{a_3(b_2-a_3b_1)-a_2(a_2b_1+b_3)}{\delta}\Big]x_0^2\\
  &&+\Big[\frac{(a_2b_1+b_3)^2+(b_2-a_3b_1)^2}{\delta^2}-1\Big]x_1^2+\frac{a_3(a_2b_1+b_3)+a_2(b_2-a_3b_1)}{\delta}x_0x_1\\
    &&+\Big[\frac{2t_2(a_2b_1+b_3)+2t_1(b_2-a_3b_1)}{\delta^2}+\frac{b_3(b_2-a_3b_1)-b_3(a_2b_1+b_3)+2a_3t_1-2a_2t_2}{\delta}\Big]x_0\\
    &&+\Big[\frac{2t_2(a_3b_1-b_2)+2t_1(a_2b_1+b_3)}{\delta^2}+\frac{(a_2b_1+b_3)b_3-b_2(a_3b_1-b_2)}{\delta}-b_1\Big]x_1\\
    &&+\frac{t_1^2+t_2^2}{\delta^2}+\frac{b_3t_1-b_2t_2}{\delta}+c_0=0.\end{eqnarray*}
    By Proposition \ref{prop5.3}, if $F=0$ then  the above equation is  an identical equation.
  Thus Equation VI has solutions   $$x=x_0+x_1\bi+x_2\bj+x_3\bk,\forall x_0,x_1\in \br,$$  where
$x_2$ and $x_3$ are given by (\ref{sx2e}) and (\ref{sx3e}).
\end{proof}

\begin{exam}\label{exam5.9}
	Consider the quadratic equation $(1+\bj)x^2 +(2\bi+\bk)x-\frac{3}{4}+\frac{3}{4}\bj= 0$.
 That is, $a=1+\bj,b=2\bi+\bk$ and $c=-\frac{3}{4}+\frac{3}{4}\bj$.
  It is obvious that $P_{ab}=0,I_c=0,P_{bc}=0,I_b+2P_{ac}=0,I_b\neq 0$. Then
$\delta=3,t_1=\frac{3}{2},t_2=0, F=0$ and $$x=x_0+x_1\bi-x_0\bj+(x_1+\frac{1}{2})\bk,\forall x_0,x_1\in \br.$$
\end{exam}

\vspace{2mm}
{\bf Acknowledgments.}\quad
 This work is supported by Natural Science Foundation of China (11871379),
 the Innovation Project of Department of Education of Guangdong Province (2018KTSCX231) and
  Key project of  National Natural Science Foundation  of Guangdong Province Universities (2019KZDXM025).

\end{document}